\documentclass[10pt]{elsarticle}

\usepackage{bm}
\usepackage[utf8]{inputenc}
\usepackage{mathrsfs}
\usepackage{xcolor}
\usepackage{graphicx}
\usepackage{subcaption}
\journal{Applied and Computational Harmonic Analysis}
\usepackage{epstopdf}
\usepackage{relsize}
\usepackage{amsmath}
\usepackage[dvips]{epsfig}
\usepackage{amsthm,amssymb}
\usepackage{wrapfig}
\usepackage{bigints}
\usepackage{setspace}
\usepackage{algorithm}
\usepackage{algpseudocode}

\usepackage{scalerel,stackengine}
\usepackage{amsfonts}
\graphicspath{{./Figures/}}
\usepackage{comment}
\usepackage{stmaryrd}
\usepackage{tikz}
\usepackage{tkz-tab}

\usepackage[textsize=tiny]{todonotes}

\newcommand{\R}{\mathbb{R}}
\newcommand{\G}{\mathcal{G}}

\newcommand{\C}{\mathbb{C}}
\newcommand{\F}{\mathcal{F}}

\newcommand{\Ecal}{\mathcal{E}}

\newcommand{\abs}[1]{\left\lvert #1 \right\rvert}

\DeclareMathOperator*{\argmin}{arg\,min}

\newtheorem{proposition}{Proposition}[section]
\newtheorem{definition}{Definition}[section]
\newtheorem{remark}{Remark}[section]
\newtheorem{theorem}{Theorem}[section]
\newtheorem{lemma}{Lemma}[section]

\begin{document}

\begin{frontmatter}

\title{Zeros of the Spectrogram of White Noise\\with Hermite Windows}

\author[ljk]{Charaf Bellourak}
\author[ljk]{Sylvain Meignen}
\author[ljk]{Kévin Polisano}
\author[ljk]{Mickaël Nahon}
\affiliation[ljk]{
organization={Jean Kuntzmann Laboratory, University Grenoble Alpes and CNRS 5225, Grenoble 38401},
country={France}
}

\begin{abstract}
    We study the zeros of the spectrogram of complex white noise analyzed with Hermite windows of order $n$, extending the classical Gaussian-window correspondence with Gaussian Analytic Functions to the polyanalytic setting. Using the polyanalytic Bargmann transform, these zeros split into two families according to the sign of the local Jacobian, giving rise to a natural notion of charge, and we show that these zeros are almost surely simple, isolated, and form a translation and rotation-invariant point process. We derive closed-form expressions for the one-point intensity, decomposed by charge, and also for the unsigned, signed, and hybrid two-point intensities. Besides, we prove that unlike the total repulsion observed with a Gaussian window, the Hermite window of order $n$ generates a pair correlation function that
    converges at the origin to a strictly positive constant depending on $n$. Finally, we propose a numerical method for detecting spectrogram zeros via the winding number around candidate local minima.
\end{abstract}


\end{frontmatter}

\section{Introduction}

The time-frequency analysis of white noise has seen significant development in recent years, particularly through the establishment of deep connections between the short-time Fourier transform (STFT) and the theory of Gaussian analytic random functions. This line of research originates in the pioneering empirical observations of Flandrin \cite{flandrin2015time}, who noticed that the zeros of the spectrogram of white Gaussian noise, computed with a Gaussian window, are distributed according to a regular pattern and proposed a filtering method based on the Delaunay triangulation of these zeros. Then a rigorous mathematical framework was proposed independently by \cite{bardenet2020zeros} and \cite{abreu2018filtering}, who both established that when the analysis window used is the unitary Gaussian functions, the spectrogram of complex white noise gives rise to a point process whose zeros coincide with those of a well-known Gaussian Analytic Function (GAF) associated with the Bargmann-Fock space, an object whose statistical structure (rotation invariance, zero repulsion, pair correlation function) has been extensively studied in the probabilistic literature, notably by Hough \cite{hough2009zeros}. This identification provided a rigorous explanation for Flandrin's empirical observations and a solid probabilistic foundation for characterizing the zeros distribution.

Bardenet and Hardy \cite{bardenet2021time} later developed a systematic correspondence between pairs (time-frequency transform, GAF) and generating functions of classical orthogonal polynomials, also extending this unifying framework to the wavelet transform and its associated hyperbolic GAF. In parallel, \cite{abreu2018filtering} extended the zero-based filtering technique to the continuous wavelet transform and provided statistical tools like pair correlation estimators for the design of such algorithms. 

However, all of theses works restrict their attention to windows leading to genuinely holomorphic random functions. This restriction naturally raises the question of its generalization: what happens when the Gaussian window is replaced by a Hermite window of order $n\geq 1$? Indeed this natural choice has been motivated in particular by the multitaper and reassignment methods discussed by Flandrin \cite{xiao2007multitaper} and by the ConceFT algorithm of Daubechies et al. \cite{daubechies2016conceft} which combines several Hermite windows to sharpen time-frequency representations. Hermite functions have also been studied as windows for Gabor frames and superframes \cite{grochenig2007gabor, grochenig2009gabor}, with recent work focusing on asymptotic frame regions associated with high-order Hermite functions \cite{faulhuber2026asymptotic}. As shown by Abreu and Shirai in a recent work \cite{abreu2026zero}, replacing the Gaussian window by a Hermite window $h_n$ of order $n$ has the effect of replacing classical GAF by a polyanalytic random functions. This work reveals a correlation structure between the zero sets of spectrograms associated with different Hermite orders, providing a rationale for the empirically observed efficiency of multitaper techniques. Nevertheless, this polyanalytic setting has so far received comparatively little attention from the point of view of intrinsic statistics of the zeros of a single spectrogram with $h_n$, as opposed to the cross-correlations between different orders $n$.

The goal of this article is twofold, and complement this recent work by focusing on the single-window case. First, we establish the theoretical framework necessary to study the zeros of the spectrogram obtained with a Hermite window of order $n$, relying on the Bargmann transform and complex Hermite polynomials. In particular, we show that these zeros, unlike in the classical holomorphic case treated in these earlier works, split into two families according to the sign of the Jacobian of the associated polyanalytic function, giving rise to a notion of charge of the zeros, either positive or negative \cite{haimi2022zeros}. Second, we develop the statistical tools needed to characterize this new point process: we compute the signed and unsigned pair correlation functions, and we derive the asymptotic behavior of the pair correlation function near the origin. Finally, motivated by practical considerations regarding the numerical detection of zeros in a spectrogram, we propose and analyze a detection method based on computing the winding number around local minima of the spectrogram, and we discuss the optimal choice of the contour radius used for this detection.

The remainder of the article is organized as follows. Section \ref{sec:background} recalls basic notions on the STFT, the Bargmann transform, and polyanalytic Fock spaces. Section \ref{polybargmann-stationari} gives the polyharmonic Bargmann transform of white Gaussian noise and some stationarity properties. Section \ref{sec:zeros} establishes the fundamental properties of the zeros of the white noise spectrogram with Hermite windows, in particular their characterization in terms of charge. Section \ref{sec:intensities} is devoted to the computation of the pair correlation function and the study of its asymptotic behavior. Section \ref{sec:detection} proposes a numerical method for zero detection and discusses the choice of the optimal contour radius. Technical proofs are deferred to the appendix.

\section{Background}\label{sec:background}

\subsection{The Short-Time Fourier Transform}
\label{stft}
Let $f, g \in L^2(\R)$, the evaluation at $(u,v) \in \R^2$ of the \emph{short-time Fourier transform} (STFT) of $f$ with window $g$ reads
\begin{equation}
\label{stft-def}
V_g f(u,v) = \int_\R f(t) \, \overline{g(t-u)} \, e^{-2i\pi v t} \, dt = \langle f, M_v T_u g \rangle,
\end{equation}
with $\langle \cdot, \cdot \rangle$ the inner product in $L^2(\R)$, $M_v f = e^{2i\pi v \cdot} f(\cdot)$, and $T_u f = f(\cdot - u)$. The squared modulus of the STFT, $\left| V_g f(u,v) \right|^2$, is called the \emph{spectrogram}. The right-hand side of (\ref{stft-def}) naturally extends the STFT to tempered distributions \cite{grochenig2001foundations}, which will be relevant later for white noise. For a fixed window $g \in L^2(\R)$,  the STFT maps $L^2(\R)$ to a closed subspace of $L^2(\R^2)$, noted $\mathcal{G}_g$, called the \emph{Gabor space}:
\begin{equation}
\label{gabor-space-def}
\mathcal{G}_g = \left\{ V_g f : f \in L^2(\R) \right\} \subset L^2(\R^2).
\end{equation}
If $\|g\|_{L^2(\R)}=1$, then $V_g$ is an isometry from $L^2(\R)$ onto $\mathcal{G}_g$. 
In what follows we denote by $\langle \cdot, \cdot \rangle_2$ on $L^2(\R^2)$. Hence $\mathcal{G}_g$ is a reproducing kernel Hilbert space with kernel
\begin{equation}
\label{rkhs-kernel}
k(u,v,u',v')=\langle M_vT_u g,\,M_{v'}T_{u'}g\rangle_2.
\end{equation}
In what follows, $g$ denotes $g(x) =  2^{1/4} e^{-\pi x^2}$ unless stated otherwise. Using that function, writing $z=u+iv$ and $w=u'+iv'$, this kernel $k_0(z,w)$ is given by \cite{abreu2010structure}:
\begin{equation}
\label{gabor-fock-kernel}
k_0(z,w)=e^{-i\pi(u'v'-uv)-\frac{\pi}{2}(|z|^2+|w|^2)}\,e^{\pi w\bar z}.
\end{equation} 

\subsection{The Bargmann Transform and Fock Space}
Recall that $\mathcal{L}^2(\C) := L^2(\C,e^{-\pi |z|^2})$ denotes the Hilbert space of all measurable functions $F: \C \rightarrow \C$ equipped with the norm \[
\|F\|_{\mathcal{L}^2(\C)}^2=\int_{\C}|F(z)|^2e^{-\pi |z|^2}\,d\mu(z)<\infty,
\]
where $\mu(z)$ is the Lebesgue measure on $\C$.
For the unitary Gaussian window $g$, the STFT admits the classical Bargmann representation \cite{bargmann1961hilbert}:
\begin{equation}
\label{def:STFT}
V_g f(\bar z) := V_g f(u,-v) = e^{\pi i u v} e^{-\frac{\pi}{2}\left| z \right|^2}  {\cal B}f(z),
\end{equation}
where ${\cal B} f$ is the Bargmann transform defined by: 
\begin{align}
    {\cal B}f(z)=2^{1/4} \int_{\R} f(t)\, e^{-\pi t^2 + 2\pi t z - \frac{\pi}{2}z^2}\, dt.
\end{align}

This identifies the STFT as an entire function of $z$. More precisely, the Bargmann transform maps $L^2(\R)$ isometrically onto the Fock space $\F(\C)$ \cite{abreu2013function}, defined as the Hilbert space of entire functions $F$ of $L^2(\C,e^{-\pi |z|^2})$. In what follows we denote 
$\langle \cdot , \cdot \rangle_{\F}$ the inner product on $\F(\C)$. This space is a reproducing kernel Hilbert space with kernel
\begin{equation}
\label{fock-kernel}
K^0(z,w)=e^{\pi z\bar w}.
\end{equation}

\begin{remark}
Using (\ref{def:STFT}) and (\ref{fock-kernel}), one easily obtains  (\ref{gabor-fock-kernel}).
\end{remark}

\subsection{Gabor space, Hermite functions and polyanalytic functions}

Let us consider the family of Hermite functions $(h_n)_{n \ge 0}$, defined by:
\begin{equation}
\label{normalized-HF}
h_n(t)=(-1)^n\frac{2^{1/4}}{\sqrt{2^n n!}} e^{\pi t^2} \frac{d^n}{dt^n} e^{-2\pi t^2}.
\end{equation}
\begin{definition} The corresponding Gabor space $\G^n$ is:
\begin{equation}
\label{gabor-hermite}
\mathcal{G}^n = \left\{ V_{h_{n}} f : f \in L^2(\R) \right\} \subset L^2(\R^2).
\end{equation}
\end{definition}
The kernel of $\mathcal{G}^n$ is given in \cite{abreu2010structure}:
\begin{equation}
\label{kernel-hermite}
\begin{split}
k_n(z, w) = \frac{1}{n!}\, &e^{i\pi(uv - u'v') + \frac{\pi}{2}(|w|^2 - |z|^2)} 
\left( \frac{d}{dw} \right)^{n} \!\! \left[ e^{\pi \bar{z} w - \pi|w|^2}(w - z)^{n} \right].
\end{split}
\end{equation}

\begin{definition}[Polyanalytic function]
A function $F:\C\to\C$ is called \emph{Polyanalytic} of order n if its $(n+1)$-th derivative with respect to $\bar{z}$ vanishes, that is, 
\begin{eqnarray}
\partial_{\bar{z}}^{n+1} F=0.
\end{eqnarray}
Equivalently, $F$ can be written as:
\begin{eqnarray}
F(z) = \sum_{k=0}^n \bar{z}^k \phi_k(z),
\end{eqnarray}
where $\{ \phi_k \}_{k=0}^n$ are analytic functions.
\end{definition}

\begin{definition}[The true polyanalytic Fock space]
A function $F:\C\to\C$ belongs to the \emph{True Polyanalytic Fock space} $\F^{n}(\C)$ of order $n$ 
if $\|F\|_{\mathcal{L}^2(\C)}^2<\infty$ and there exists an entire function $\phi$ such that:
    \begin{equation}
    \label{polyanalytic-eq-diff}
    F(z)=\left(\frac{\pi^n}{n!}\right) ^\frac{1}{2} \ e^{\pi |z|^2} \left( \frac{d}{dz} \right)^{n} \left(e^{-\pi |z|^2} \phi(z) \right).
    \end{equation}
\end{definition}
\begin{remark}
\begin{enumerate}
    \item $n=0$ corresponds to analytic functions.
    \item Every function of $\F^{n}(\C)$ is polyanalytic of order $n$.
\end{enumerate}
\end{remark}

\begin{definition}
The true polyanalytic Bargmann transform of order $n$ of a function $f\in L^2(\R)$ is defined by:
\begin{equation}
    (\mathcal{B}_{n} f)(z) = \left(\frac{\pi^n}{n!}\right) ^\frac{1}{2} \ e^{\pi |z|^2} \left( \frac{d}{dz} \right)^{n} \left(e^{-\pi |z|^2} \mathcal{B} f\right)(z).
\end{equation}
\end{definition}
Using (\ref{polyanalytic-eq-diff}) and the fact that $\mathcal{B}f$ is an entire function, it is clear that $\mathcal{B}_{n} f$ ($n \geq 1$) is a polyanalytic function. The  
STFT computed with $h_n$ and the polyanalytic Bargmann transform are related through:
\begin{equation}
\label{poly-bargmann}
(\mathcal{B}_{n} f)(z)=e^{-i\pi uv+\frac{\pi}{2}|z|^2} \,(V_{h_n}f)(u,-v),
\end{equation}
which defines a mapping between $L^2(\R)$ and $\F^{n}(\C)$ \cite{abreu2010structure, abreu2010sampling}.

The multiplier operator 
\begin{equation}
\label{multiplier-n}
(\mathcal{E}_nF)(z)=e^{i\pi uv+\frac{\pi}{2}|z|^2}F(z)
\end{equation}
defines an isometric isomorphism  $\mathcal{G}^n \cong \F^{n}(\C)$.
\begin{proposition}
The space $\F^n(\C)$ is a reproducing kernel Hilbert space and its kernel $K^n(z,w)$ is the following:
\begin{equation}
\label{kernel-poly-gabor}
    K^n(z,w) = e^{-i\pi(uv - u'v') + \frac{\pi}{2}(|w|^2 + |z|^2)} k_n(z,w).
\end{equation} 
\end{proposition}

\begin{proof}
Let $f\in \mathcal{G}^n$ and $\psi = \Ecal_n f \in \F^n(\C) $, the reproducing property gives 
$f(z) = \langle f, k_n(z,\cdot) \rangle_{\mathcal{G}^n}$.
As $\Ecal_n$ is unitary, 
\begin{eqnarray*}
\begin{aligned}    
\langle f, k_n(z, \cdot) \rangle_2 &=& \langle \Ecal_n f, \Ecal_n k_n(z, \cdot) \rangle_{\F}
    =\langle \psi, \Ecal_n k_n(z, \cdot) \rangle_{\F}.
\end{aligned}
\end{eqnarray*}
Then, $ e^{-i\pi uv-\frac{\pi}{2}|z|^2} \psi(z) = \langle \psi,e^{i\pi u'v'+\frac{\pi}{2}|w|^2}  k_n(z, \cdot) \rangle_{\F}$
meaning that $$ \psi(z)=\langle \psi, K_n(z, \cdot) \rangle_{\F_n(\C)}.$$ 
\end{proof}
Combining (\ref{kernel-poly-gabor}) with (\ref{kernel-hermite}), one obtains:
\begin{eqnarray}
\label{poly-fock-kernel1}
\begin{aligned}
K^n(z,w) = \frac{1}{n!}\,e^{\pi|w|^2}\left( \frac{d}{dw} \right)^{n} \! \left[ e^{\pi \bar{z} w - \pi|w|^2}(w - z)^{n} \right]
=L_n(\pi|z-w|^2)\,e^{\pi w\bar z},
\end{aligned}
\end{eqnarray}
where $L_n$ is the Laguerre polynomial of degree $n$.

Finally, we prove the following proposition: 
\begin{proposition} 
\label{prop:Hkn}
The \emph{STFT} of an Hermite function $h_k$ with the window $h_n$ is given by:
\begin{equation}
\label{stft-hermite-hermite}
V_{h_n}h_k(\bar{z}) = e^{i \pi u  v} e^{-\frac{\pi}{2} |z|^2} H_{k,n}(\sqrt \pi \ z, \sqrt \pi \ \bar{z}),
\end{equation}
where $H_{k,n}$ stands for the complex Hermite polynomials defined as \cite{abreu2026zero}:
\begin{equation*}
H_{k,n}(z,\bar z)=
\begin{cases}
\sqrt{\dfrac{n!}{k!}}\, z^{k-n} L_n^{(k-n)}\!\bigl( |z|^2\bigr),  k>n,\\[6pt]
\sqrt{\dfrac{k!}{n!}}\, (-\bar z)^{\,n-k} L_k^{(n-k)}\!\bigl( |z|^2\bigr),  k\le n,
\end{cases}
\end{equation*}
with $L_n^{(\alpha)}$ the generalized Laguerre polynomial.
\end{proposition}
The proof is available in Appendix \ref{AppendixA}.

\section{Polyharmonic Bargmann transform of white Gaussian noise and stationarity properties}
\label{polybargmann-stationari}
\subsection{Polyharmonic Bargmann transform of white Gaussian noise}
A \emph{complex Gaussian white noise} (CGWN) 
can be viewed as a random tempered distribution, denoted by  ${\cal S}'(\R)$,  and its Hermite coefficients are i.i.d. standard complex Gaussian random variables \cite{abreu2026zero}. More precisely, since $(h_k)_{k\ge 0}$  is an orthonormal basis of $L^2(\mathbb{R})$, then the decomposition of white noise into this basis can be written in ${\cal S}'(\R)$ as
\begin{equation}
\label{white-noise-infty}
    \xi= \sum_{k=0}^{\infty} \xi_k h_k,
\end{equation}
where the random variables $\xi_k:=\langle \xi,h_k\rangle$ are i.i.d.\ with law ${\mathcal{C} \mathcal{N}}(0,1)$. As the convergence takes 
place in ${\cal S}'(\R)$, we may write: 
\begin{eqnarray}
\label{def:STFT_hermite}
\begin{aligned}
V_{h_n}\xi\left(\frac{\bar{z}}{\sqrt \pi}\right) &= \langle \xi, M_{-v} T_u h_n \rangle = \sum_{k=0}^{\infty} \xi_k V_{h_n} h_k\left(\frac{\bar{z}}{\sqrt \pi}\right)\\
&= e^{i u v} e^{-\frac{|z|^2}{2} } \sum_{k=0}^{\infty} \xi_k H_{k,n}( z,\ \bar{z})
:= e^{i  u v} e^{-\frac{|z|^2}{2}} f_n(z,\bar z).
\end{aligned}
\end{eqnarray}
So using (\ref{poly-bargmann}), we have $({\cal B}_n\xi)(\frac{z}{\sqrt \pi}) = f_n(z,\bar z)$. 
One also has the following relation \cite{abreu2026zero}:
\begin{equation}
\label{Landau_level}
f_{n+1}(z,\bar z) = \frac{1}{\sqrt{n+1}} 
(\partial_z f_n(z,\bar z) - \bar z f_n(z,\bar z)).
\end{equation}

For the sake of simplicity, from now on we will  put 
$f_n(z):= f_n(z,\bar z)$.
\subsection{Stationarity with respect to translation}
\label{sec:stab_trans}
Consider the kernel associated with ${\cal B}_n(z+a)$, we immediately get: 
\begin{eqnarray*}
K^n(z+a,w+a) &=& L_n(\pi|z-w|^2)\,e^{\pi (w+a)\overline{z+a}}= K^n(z,w) \,e^{\pi (a \bar z + w \bar a + |a|^2)}.
\end{eqnarray*}
Defining $h_n(z) := {\cal B}_n (z) e^{\pi (z \bar a+ \frac{|a|^2}{2})}$, 
one gets: 
\begin{eqnarray*}
\mathbb{E}[h_n(w) \overline{h_n(z)}] = K^n(z,w) \,e^{\pi (a \bar z + w \bar a + |a|^2)}.
\end{eqnarray*}
Thus we have 
\begin{eqnarray}
\label{def:stab_translat}
{\cal B}_n (z+a) := 
 e^{\pi (z \bar a+ \frac{|a|^2}{2})} {\cal B}_n (z).
 \end{eqnarray}
The prefactor of the right-hand side is nonzero, thus the zeros of ${\cal B}_n(z)$ are invariant with respect to translation.
\subsection{Stationarity with respect to rotation}
\label{sec:stab_rot}
 Remembering  that 
 $$
 f_n(z,\bar z) = \sum_{k=0}^n \binom{n}{k} f_0^{(k)}(z) (- \bar z)^{n-k}.
 $$
 Thus, 
 \begin{eqnarray*}
 f_n(e^{i\theta} z, e^{-i\theta}\bar z) = 
 \sum_{k=0}^n \binom{n}{k} f_0^{(k)}(e^{i\theta}z) (- e^{-i\theta} \bar z)^{n-k}\\
 = \sum_{k=0}^n \binom{n}{k} \sum_{p \geq k} \xi_p \frac{\sqrt{p!}}{(p-k)!} z^{p-k} e^{i\theta(p-k)} (- e^{-i\theta} \bar z)^{n-k}\\
 = \sum_{k=0}^n \binom{n}{k} \sum_{p \geq k} \xi_p e^{i\theta(p-n)}\frac{\sqrt{p!}}{(p-k)!} z^{p-k} (\bar z)^{n-k},
 \end{eqnarray*}
 therefore $f_n(e^{i\theta} z) \stackrel{d}{=} f_n(z)$. 
\section{Zeros of White Noise with Hermite Functions}\label{sec:zeros}

\subsection{Zeros of White Noise as a Point Process}

The polyanalytic function $f_n$ has the same zeros as 
$V_{h_n} \xi$, thus we study the zeros of the former function. We already know that when $n=0$, the zeros 
are isolated \cite{bardenet2021time}. 
Then using \eqref{Landau_level}, one can write that if $n= 0$ the zeros $z_0$ of $f_1$ 
correspond to:
\begin{eqnarray*}
\partial_z f_0(z_0) - \bar z_0 f_0(z_0) = 0.
\end{eqnarray*}
Now assume that $f_0(z_0) \neq 0$, then $f_0$ is holomorphic at $z_0$ and we may 
write: 
\begin{eqnarray*}
\bar z_0 := A(z_0) =  \frac{\partial_z f_0(z_0)}{f_0(z_0)},
\end{eqnarray*}
with $A$ holomorphic at $z_0$ at thus defining $\overline A$, the conjugate of $A$, we may write $z_0 = \overline{A} (\bar z_0) = \overline{A} \circ A(z_0)$. Then, as $\overline{A} \circ A$ is holomorphic at $z_0$, then $z_0$ is a fixed point of an holomorphic function and is isolated
provided  $\overline{A} \circ A$ is not the identity, which is true with probability $1$. 
Then if $f_0(z_0) = 0$ then $\partial_z f_0(z_0) = 0$ and thus $z_0$ is a second order zero of an holomorphic function and is therefore isolated. 

In the general case, $n \geq 1$, let us consider $f_n(z,\bar z) = u(z,\bar z) + i v(z,\bar z)$, 
then the zeros correspond to $f_n(z_0,\bar z_0) = 0$
and, omitting $(z_0,\bar z_0)$ for the sake of simplicity:
\begin{eqnarray}
\label{def:det}
\det f_n'(z_0) = \left |\begin{array}{c c} \partial_x u & \partial_y u\\
     \partial_x v & \partial_y v
     \end{array} \right | = |\partial_z f_n|^2 -  |\partial_{\bar z} f_n|^2 = 0.
\end{eqnarray}
So viewing $f_n(z,\bar z)$ as a bidimensional function, it is clear that when \eqref{def:det} is 
satisfied, there may exist a curve of zeros around $z = z_0$, and $z_0$ is called a degenerate zero. 
On the contrary, when \eqref{def:det} is not satisfied, the zeros are isolated.

The Bulinskaya-type lemma implies that, for the random field $f_n$, 
degenerate zeros occur with probability zero, therefore the zeros are almost surely isolated and simple even for $n\ge 1$. For the sake of consistency we recall \cite{abreu2026zero, armentano2025general}:

\begin{lemma}[Bulinskaya-type lemma]
\label{lem:point_process}
Let $U \subset \mathbb{R}^2$ be open, and let
$F : U \to \mathbb{R}^2$
be a random field with almost surely $C^1$ sample paths. 
Fix $v \in \mathbb{R}^2$. Assume that for every relatively compact open set $O \subset U$, there exists a neighborhood $V$ of $v$ and a constant $C_O$ such that for every $y \in V$,
\begin{equation}
\int_O p_{F(x)}(y)\,dx \le C_O,
\end{equation}
where $p_{F(x)}$ denotes the density of $F(x)$ on $\R^2$.
Then
\begin{equation}
\mathbb{P}\Bigl(\exists\, x \in O : F(x)=v \text{ and } \det F'(x)=0\Bigr)=0.
\end{equation}
In other words, with probability one, the level $v$ is not a critical value of $F$.
\end{lemma}
A direct application of this lemma consists of considering the random field  
$F_n(z)=\bigl(\Re f_n(z,\bar z),\, \Im f_n(z,\bar z)\bigr) \in \R^2$,
which is almost surely $C^1$ for each noise realization and admits a density $p_{F_n(z)}$ on $\R^2$ satisfying 
\begin{eqnarray*}
\int_O p_{F(z)}(y)\,dz \le \frac{\abs O}{\pi}
\end{eqnarray*}
for every relatively compact set $O\subset \C$ and every $y\in \R^2$ (see \cite[Proposition 4.8]{abreu2026zero}). 

Therefore the assumptions of the lemma are satisfied and, taking $v=0$, one concludes that the zeros of $f_n(z, \bar{z})$ are almost surely simple and isolated, and thus define a point process, which can be studied using statistical tools like \emph{pair correlation function}. 

When $n \geq 1$, there exists three types of 
zeros depending on the sign of the Jacobian of $f_n$
defined by 
\begin{equation}
\label{def:Jac}    
J_{f_n}(z_0) = |\partial_z f_n(z_0)|^2 - |\partial_{\bar{z}} f_n(z_0)|^2.
\end{equation}
If $J_{f_n}(z_0) >0$ (resp. $<0$), we say that the zero is charged positively (resp. negatively). 
When it is null, the zero is said to be degenerate, a situation that happens with a null probability since the zeros are point processes. 
\subsection{On Maxima and Zeros of $|V_{h_0}\xi|^2$ and $|V_{h_1}\xi|^2$}

We investigate the relation between maxima at order $0$ and zeros at 
order $1$. If $z_0$ is a zero of $f_1$, considering $\phi_0$ 
defined as:
\begin{equation}
    \phi_0(z) := |f_0(z)|^2 e^{-|z|^2}.
\end{equation}
A direct computation, using the fact $f_0$ is holomorphic, gives: 
\begin{equation}
\label{partial_barz}
    \partial_{z}\phi_0(z)=f_1(z) \,\overline{f_0(z)} e^{-|z|^2}.
\end{equation}
Therefore, the critical points of $\phi$ are exactly the points where $f_0(z)=0$ or $f_1(z)=0$.
$f_0(z) = 0$ obviously corresponds to minima of $\phi_0$. To determine whether critical points corresponding to $f_1(z)=0$ are maxima or saddle points, we look into the Hessian matrix $H(z)$ of $\phi_0(z)$. In fact, we have the following:
\begin{itemize}
    \item If $\det H(z)>0 \text{ and } \partial_z\partial_{\bar z}\phi_0(z)<0$, then there is  a local maxima at $z$.
    \item If $\det H(z)<0$, then there is a saddle point at $z$.
\end{itemize}
Using (\ref{partial_barz}), we get:
\begin{eqnarray*}
\begin{aligned}
    \partial_z\partial_{\bar z}\phi_0 (z) &= e^{-|z|^2} \Big(|f_1(z)|^2 + f_0(z)\overline{\partial_{\bar z}f_1(z)}\Big)\\
    &= e^{-|z|^2}\Big(|f_1(z)|^2 - |f_0(z)|^2\Big)= - e^{-|z_0|^2} \ |f_0(z_0)|^2 < 0,
\end{aligned}
\end{eqnarray*}
 in which we used the fact that $f_0$ is 
 holomorphic to write $\partial_{\bar z} f_1 = -f_0$. 
For the determinant of the Hessian matrix, one has
\begin{eqnarray}
\label{det_hessian}
\det H(z)=4\Big[(\partial_z\partial_{\bar z}\phi_0(z))^2-|\partial_z^2\phi_0(z)|^2\Big].
\end{eqnarray}
We then remark that 
\begin{equation*}
\label{partial2_phi}
    \partial^2_z\phi_0 (z) = e^{-|z|^2} \ \overline{f_0(z)} \ \big(\partial_z f_1(z)-\bar z\,f_1(z)\big).
\end{equation*}
Now consider $z_0$ such that $f_1(z_0)=0$ and $f_0(z_0) \neq 0$, one obtains
\begin{eqnarray*}
\partial_z^2\phi_0(z_0)=e^{-|z_0|^2}\,\overline{f_0(z_0)}\,\partial_z f_1(z_0),
\end{eqnarray*}
and thus \eqref{det_hessian} becomes
\begin{eqnarray}
\begin{aligned}
\det H(z_0) =& 4e^{-2|z_0|^2}|f_0(z_0)|^2\Big(|f_0(z_0)|^2-|\partial_z f_1(z_0)|^2\Big)\\
=& 4e^{-2|z_0|^2}|f_0(z_0)|^2\Big( |\partial_{\bar z} f_1(z_0)|^2 - |\partial_{z} f_1(z_0)|^2\Big)\\
=&-4e^{-2|z_0|^2}|f_0(z_0)|^2 \ J_{f_1}(z_0),
\end{aligned}
\end{eqnarray}
and thus the local maxima (resp. saddle points) of $|V_{h_0}\xi|^2$ 
are the zeros of $f_1$ charged negatively (resp. positively) as illustrated in Fig. \ref{Fig1}. 
Unfortunately, there is no extension to $n \geq 1$, 
because $f_n$ is no longer holomorphic.

\begin{figure*}
\begin{center}
\begin{tabular}{cc}
\includegraphics[scale = 0.45]{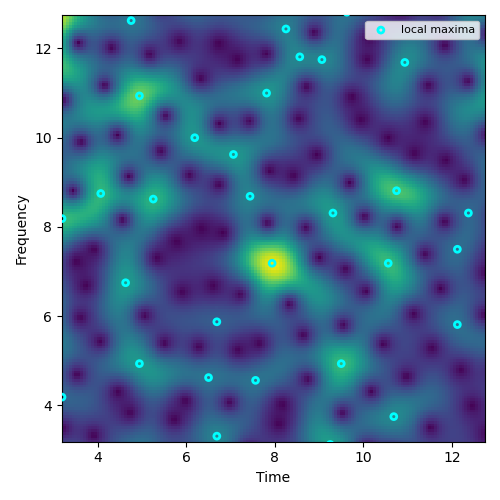}&\includegraphics[scale = 0.45]{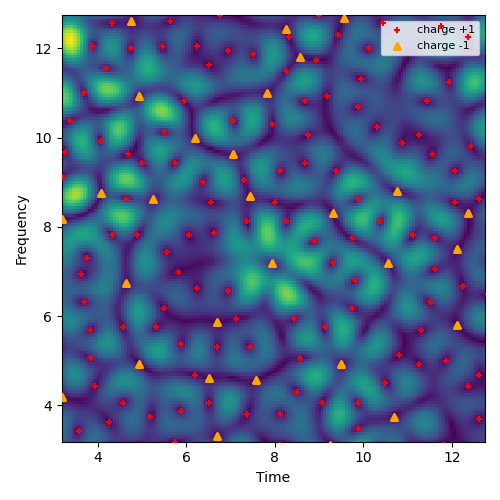}\\
(a)&(b)
\end{tabular}
\end{center}
\caption{(a):  local maxima of $|V_{h_0}\xi|^2$. (b): positive and negative zeros of  $|V_{h_1} \xi|^2$ }
\label{Fig1}
\end{figure*}

\section{Computation of the one and two points intensities}\label{sec:intensities}

Using Kac-Rice Formula, the $k$-point intensity of $f_n$ is defined as follows:
\begin{eqnarray}
\label{def:k_intensity}
\begin{aligned}
\rho_{n}^{[k]}(a_1,a_2,\ldots,a_k) 
 =\\
\mathbb{E}[|J_{f_n(a_1)}\ldots J_{f_n(a_k)}|  
\mid f_n(a_1)=0, \ldots, f_n(a_k)=0] p_{f_n(a_1),\ldots,f_n(a_k)} (0,\ldots,0).
\end{aligned}
\end{eqnarray}
As $f_n$ is stationary, we had rather study its simpler form, setting $a_1$ to zero: 
\begin{eqnarray}
\begin{aligned}
\rho_n^{[k]}(0,a_2,\ldots,a_k) 
 =\\
\mathbb{E}[|J_{f_n(0)}\ldots J_{f_n(a_k)}|  
\mid f_n(0)=0, \ldots, f_n(a_k)=0] 
p_{f_n(0),\ldots,f_n(a_k)} (0,\ldots,0).
\end{aligned}
\end{eqnarray}
So if $k=1$,
we obtain the 1-point intensity which is 
constant, and multiplied by $\pi$ we get the average 
number of zero on circles of radius 1 on $|V_{h_n} \xi|^2$. If $k=2$, and multiplying by $\pi$ (see normalization in \eqref{def:STFT_hermite}), 
$\rho_n^{[2]}$ is related to the probability of  
having 2 zeros of $|V_{h_n} \xi|^2$ in a discus of radius $r$ in the complex plane. 
\subsection{Computation of the First Intensity of Hermite Zeros}
Though a proof of the calculation of the 1-point intensity already exists 
in \cite[Theorem 1.8]{haimi2022zeros}. We propose here a much simpler version, based on a different approach. 
To compute the 1-point intensity we recall that 
\begin{eqnarray*}
f_0(z) &=& \sum_{k} \xi_k \frac{z^k}{\sqrt{k!}},\\
f_n(z) &=& \frac{1}{\sqrt{n !}} 
(\partial_z - \bar z)^n f_0(z)= \frac{1}{\sqrt{n!}} \sum_{k=0}^n \binom{n}{k} f_0^{(k)}(z) (-\bar z)^{n-k}.\\
\end{eqnarray*}
Then the Jacobian of $f_n(z)$ reads $J_{f_n}(z) = |\partial_z f_n(z)|^2 -  |\partial_{\bar z} f_n(z)|^2$. 
Because of the stationarity of the process $f_n$ with respect to translation (see Sec. \ref{sec:stab_trans}), one can compute the 1-point intensity using Kac-Rice formula 
at zero. We first remark that: 
\begin{eqnarray*}
J_{f_n}(0)= \left|\frac{f_0^{(n+1)}(0)}{\sqrt{n!}}\right|^2 -\left|n \frac{f_0^{(n-1)}(0)}{\sqrt{ n!}}\right|^2=(n+1) |\xi_{n+1}|^2-n |\xi_{n-1}|^2,\\
\end{eqnarray*}
and then that 
\begin{eqnarray*}
p_{f_n(z)}(y) = \frac{1}{\pi
K^n\left(\frac{z}{\sqrt \pi}, \frac{z}{\sqrt \pi}\right)}
e^{-\frac{|y|^2}{K^n\left(\frac{z}{\sqrt \pi}, \frac{z}{\sqrt \pi}\right)}},
\end{eqnarray*} 
we obtain using the Kac-Rice formula 
that: 
\begin{eqnarray}
\label{def:Kac-Rice}
\begin{aligned}
\rho_n^{[1]} = {\mathbb E} [|J_{f_n}(0)| \ | f_n(0) = 0] p_{f_n(0)}(0)=\frac{{\mathbb E} [|J_{f_n}(0)| \  | f_n(0) = 0]}{\pi}.
\end{aligned}
\end{eqnarray}
Then, we remark that $f_n(0) = \frac{f_0^{(n)}(0)}{\sqrt{ n!}} = \xi_n$. Thus, $J_{f_n}(0)$ is independent 
from $f_n(0)$. Consequently:
\begin{eqnarray*}
\rho_n^{[1]}= \frac{{\mathbb E} [|J_{f_n}(0)|]}{\pi}= \frac{{\mathbb E} [ |(n+1) |\xi_{n+1}|^2-n |\xi_{n-1}|^2|]}{\pi}\\
= \frac{1}{\pi} \left \{  (n+1) {\mathbb E} [|\xi_{n+1}|^2]+ n {\mathbb E}[|\xi_{n-1}|^2] - 2 {\mathbb E} [\min ((n+1)|\xi_{n+1}|^2,n |\xi_{n-1}|^2)] \right \}. 
\end{eqnarray*}
Using the fact that, for any $n$, $\xi_{n}$ is ${\cal C \cal N} (0,1)$, its real and imaginary parts are 
${\cal N} (0,\frac{1}{2})$, and thus $2 |\xi_{n}|^2$ is $\chi_2$ distributed with 2 degrees of freedom and 
thus $ {\mathbb E}[|\xi_{n}|^2] = {\mathbb E}[|\xi_{n+1}|^2]=1$. 
Then, using properties on the expectation of the minimum of weighted $\chi_2$ laws, we may write: 
\begin{equation*}
\label{def:minchi2}
{\mathbb E} [\min ((n+1)|\xi_{n+1}|^2,n |\xi_{n-1}|^2)] = \frac{(n+1)n}{2n+1}.
\end{equation*}
Putting everything together we obtain: 
\begin{eqnarray}
\label{rho1}
\rho_n^{[1]} &=& \frac{1}{\pi} \left(n+ \frac{1}{2} + \frac{1}{4n+2}\right).
\end{eqnarray}

Let us then derive the 1-point intensities of the positive and negative charges. In what follows, we denote $\rho^{[1],+}_n$ 
(resp. $\rho^{[1],-}_n$), the 1-point intensity of the zeros charged positively (resp. negatively), namely $f_n(z_0) = 0$ with 
$J_{f_n}(z_0) > 0$ (resp. $J_{f_n}(z_0) < 0$). For that we use (\ref{def:Kac-Rice}) but with the signed Jacobian:
\begin{eqnarray}
\label{def:signed_kac}
\begin{aligned}
\rho^{[1],+}_n - \rho^{[1],-}_n = \frac{{\mathbb E} [J_{f_n}(0) \  |  \ f_n(0) = 0]}{\pi n!} = \frac{{\mathbb E} [ (n+1) |\xi_{n+1}|^2-n |\xi_{n-1}|^2]}{\pi}  = \frac{1}{\pi}.\\
\end{aligned}
\end{eqnarray}

Using the fact that: 
\begin{eqnarray*}
\rho^{[1],+}_n + \rho^{[1],-}_n&=& \rho^{[1]}_n,
\end{eqnarray*}
the 1-point intensities of positive and negative charges are given respectively by:
\begin{eqnarray*}
\rho^{[1],+}_n = \frac{1}{2\pi} \left(n+ \frac{3}{2} + \frac{1}{4n+2}\right), \ 
\rho^{[1],-}_n = \frac{1}{2\pi} \left(n- \frac{1}{2} + \frac{1}{4n+2}\right).
\end{eqnarray*}

\subsection{Computation of the 2-point intensity  of Hermite zeros}
\label{sec:second_intensity}
The $2$-points intensity of $f_n$ satisfies 
the following property: 
\begin{eqnarray*}
\rho_n^{[2]}(0,a_2)
 =
\mathbb{E}[|J_{f_n(0)}J_{f_n(a_2)}| 
\mid f_n(0)=f_n(a_2)=0]
p_{f_n(0),f_n(a_2)} (0,0)\\
= 
\mathbb{E}[|J_{f_n(0)}J_{f_n(e^{-i \theta_2}a_2),}| \mid f_n(0)=f_n(e^{-i \theta_2}a_2)=0] 
p_{f_n(0),f_n(e^{-i \theta_2}a_2)} (0,0)\\
=\mathbb{E}[|J_{f_n(0)}J_{f_n(r)}| \mid f_n(0)=f_n(r)=0] 
p_{f_n(0),f_n(r)} (0,0) = \rho_n^{[2]}(0,|a_2|).
\end{eqnarray*}
We therefore put $\rho_n^{[2]}(r)$ for $r \in \mathbb{R}_*^+$ ($r$ can be viewed as $|a_2|$). We first define 
\begin{eqnarray}
\label{def:vector}
\begin{aligned}
X = \left ( 
\begin{array}{c}
f_n(0)\\
f_n(r)\\
\end{array}
\right ), \ 
Y = \left ( 
\begin{array}{c}
\partial_z f_n(0)\\
\partial_{\bar z} f_n(0)\\
\partial_z f_n(r)\\
\partial_{\bar z} f_n(r)
\end{array}
\right ).
\end{aligned}
\end{eqnarray}
Both $X$ and $Y$ are zero mean Gaussian vectors. In that case, 
$Z = (Y \mid X=0)$ follows a Gaussian law with expectation vector and covariance matrix defined by:
\begin{eqnarray}
\label{def:cond_law}
\begin{aligned}
\mathbb{E}[Z]  &= \mathbb{E}[Y] = 0_{{\mathbb R}^4}, \\
\Gamma &= \Sigma_Y - (\Sigma_0)^H (\Sigma_X)^{-1} \Sigma_0,\\
\end{aligned}
\end{eqnarray}
with $\Sigma_Y = \mathrm{Cov}(Y)$, $\Sigma_X = \mathrm{Cov}(X)$, $\Sigma_0 =\mathrm{Cov}(X,Y)$, 
and, $H$ denotes the Hermitian transpose. 
To compute the 2-point intensity we also need:
\begin{eqnarray}
\label{density_rho2}
p_{f_n(0),f_n(r)} (y) = \frac{1}{\pi^2 \det \Sigma_X} 
e^{-y^* (\Sigma_X)^{-1} y}.
\end{eqnarray}
\begin{figure*}
\begin{center}
\begin{tabular}{cc}
\includegraphics[scale = 0.39]{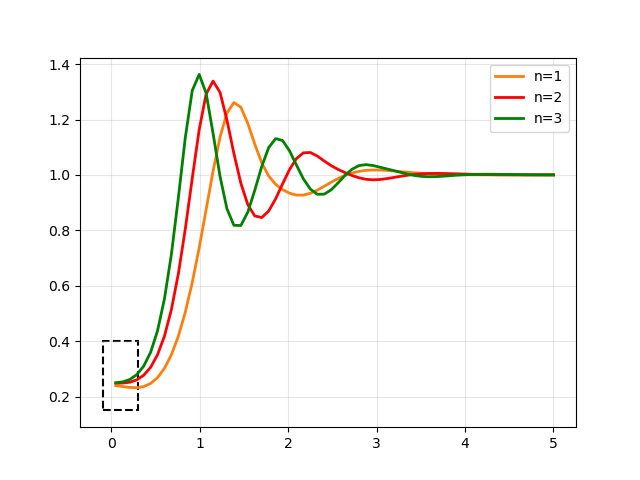}&\includegraphics[scale = 0.39]{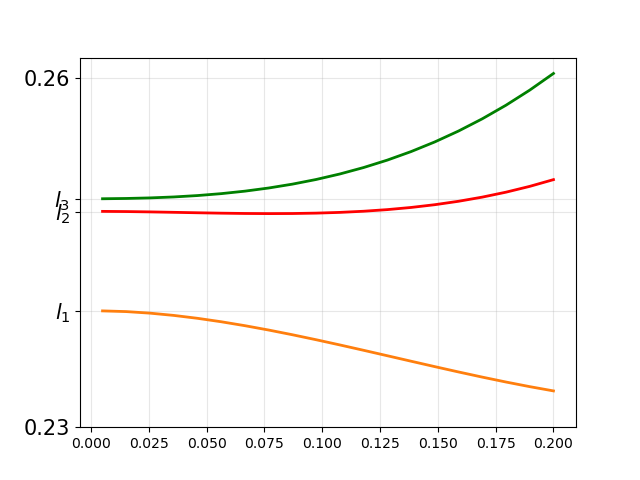}\\
(a)&(b)
\end{tabular}
\end{center}
\caption{(a): $g_n^{[2]} (r)$ for $n=1,2$ and $3$ (b): Zoom of the dashed black square, highlighting $l_1$, $l_2$ and $l_3$.}
\label{Fig2}
\end{figure*}
Then, let $\Tilde{Z}$ be a centered complex Gaussian vector in $\C^d$ with covariance matrix $\Gamma \succeq 0$, and let $A \in {\cal M}_d(\C)$ be Hermitian. Then for every $t \in \R$ \cite{turin1960characteristic},
\begin{equation}
\label{variable_determinant}
    \mathbb E \left[ \exp(it\Tilde{Z}^* A \Tilde{Z}) \right ] = \det(I - i t \Gamma A)^{-1}.
\end{equation}
Let us now apply this result to the centered complex Gaussian vector $Z$ with covariance 
$\Gamma \succeq 0$ and the hermitian matrix $\Lambda$, defined by 
$$\Lambda(s,t)=s A_1+ t A_2,$$
with $A_1 = \mathrm{Diag}(1,-1, 0,0)$ and $A_2 = \mathrm{Diag}(0,0, 1,-1)$, 
and $s,t \in \R$. 
Then we introduce 
\begin{eqnarray*}
    J_0 := Z^* A_1 Z, \
    J_r := Z^* A_2 Z.
\end{eqnarray*}
With these notations we can express the 2-point intensity as
\begin{eqnarray}
\label{def:2points}
\begin{aligned}
\rho_n^{[2]} (r):= \mathbb{E}[|J_{f_n(0)}J_{f_n(r)}| \mid f_n(0)=f_n(r)=0] p_{f_n(0),f_n(r)} (0,0)
= \frac{ \mathbb{E}[|J_0 J_r|] }{\pi^2 \det \Sigma_X}.
\end{aligned}
\end{eqnarray}
Then defining the function $\varphi$ as: 
\begin{eqnarray}
\label{eq:def_phi}
\begin{aligned}
\varphi(s,t) &= \mathbb{E}\ \left[\exp \bigl(i Z^* \Lambda(s,t) Z \bigr) \right]\\
&= \mathbb{E}\ \left[ \exp\bigl( i(sJ_0+tJ_r) \bigr ) \right]
= \det \ \bigl(I-i\ \Gamma \ \Lambda(s,t)\bigr)^{-1},
\end{aligned}
\end{eqnarray}
one has the following proposition:
\begin{proposition}
\label{unsigned_rho2}
\begin{eqnarray}
\label{def:rho2}
\rho_n^{[2]}(r) = \frac{4}{\pi^2}
\int_0^\infty \int_0^\infty
\frac{\psi (s,t)}{s^2t^2}  ds dt  \times \frac{1}{\pi^2 \det \Sigma_X},
\end{eqnarray}
where 
\begin{eqnarray*}
\psi(s,t)= 1-\Re\varphi(s,0)-\Re\varphi(0,t)
+\frac{1}{2}\bigl(\Re\varphi(s,t)+\Re\varphi(s,-t)\bigr),
\end{eqnarray*}
with $\varphi$ defined in \eqref{eq:def_phi}.
\end{proposition}
The proof is available in Appendix \ref{AppendixB}. The practical calculation of \eqref{def:rho2} is available in Appendix \ref{AppendixC}. The pair correlation function $g_n^{[2]}$ is the normalization of $ \rho_n^{[2]}$ by $\rho_n^{[1]} \times \rho_n^{[1]} $: 
\begin{eqnarray*}
g_n^{[2]}(r)=\frac{\rho_n^{[2]}(r)}{\rho_n^{[1]}\rho_n^{[1]} }.
\end{eqnarray*}

We remark that the functions $g_1^{[2]}$, $g_2^{[2]}$  and $g_3^{[2]}$ 
tend to a nonzero constant when $r$ tends to zero as illustrated in Fig. \ref{Fig2} (a), unlike in the holomorphic case $g_0^{[2]}$. Indeed we have the following proposition:
\begin{proposition}
\label{limit_pair_function}
The pair correlation function $g_n^{[2]}(r)$ converges to a 
nonzero constant when $r\rightarrow 0$, and for $n\ge 0$, one has:
\begin{eqnarray*}
\lim_{r\rightarrow0} g_n^{[2]}(r) = \frac{n(n+1)(n^2+n+1)}{(2n^2+2n+1)^2}.
\end{eqnarray*}
\end{proposition}
The proof is available in Appendix \ref{AppendixD}. Finally, defining the sequence $(l_n)_{n\ge0}$ defined for $n\ge0$:
\begin{eqnarray*}
l_n := \frac{n(n+1)(n^2+n+1)}{(2n^2+2n+1)^2},
\end{eqnarray*}
we immediately get that $l_n$ tends to $\frac{1}{4}$. Moreover, the sequence $(l_n)_{n\ge0}$ is strictly increasing.
The proof is available in Appendix \ref{AppendixE}, an illustration of this result, being given in Fig. \ref{Fig2} (b).
\subsection{Signed 2-point intensity}
\label{sec:signed_2points}
In this section, we are interested in computing the signed 2-point intensity: 
\begin{eqnarray}
\label{def:signed_2points}
\begin{aligned}
\tilde \rho_n^{[2]} (r):= \mathbb{E}[J_0 J_r] p_{f_n(0),f_n(r)} (0,0)=\frac{\mathbb{E}[\omega(Z)]}{\pi^2 \det \Sigma_X},
\end{aligned}
\end{eqnarray}
with 
$$
\omega(Z) = (|z_1|^2-|z_2|^2)(|z_3|^2-|z_4|^2),
$$
for which have the following expression: 
\begin{proposition}
\label{signed_rho2}
\begin{eqnarray*}
\tilde \rho_n^{[2]} (r) =
\frac{(\Gamma_{11} - \Gamma_{22})(\Gamma_{33} - \Gamma_{44}) 
+|\Gamma_{13}|^2 - |\Gamma_{14}|^2-|\Gamma_{23}|^2 + |\Gamma_{24}|^2}{\pi^2 \det \Sigma_X}.
\end{eqnarray*}
where $\Gamma$ is defined in (\ref{def:cond_law}).
\end{proposition}
The proof is available in Appendix \ref{AppendixF}. 
\subsection{Hybrid 2-point intensity}
\label{sec:hybrid_2points}
Let us now compute the hybrid 2-point intensity:
\begin{eqnarray*}
\hat{\rho}_n^{[2]}  (r):= \mathbb{E}[|J_0| J_r] p_{f_n(0),f_n(r)} (0,0)= \frac{\mathbb{E}[|J_0| J_r]}{\pi^2 \det \Sigma_X}.
\end{eqnarray*}
We have showed in (\ref{sec:stab_trans}) and (\ref{sec:stab_rot}) that the zeros of $f_n$ are invariant by translation and rotation. Then defining the two following transformations :
\begin{eqnarray*}
T(z)=z-r, \quad R(z) = e^{i\pi}z,
\end{eqnarray*}

we have:
\begin{eqnarray*}
\mathbb{E}[|J_0| J_r]= \mathbb{E}[|J_{T(0)}| J_{T(r)}]= \mathbb{E}[|J_{-r}| J_0] 
= \mathbb{E}[|J_{R(-r)}| J_{R(0)}] 
= \mathbb{E}[|J_r| J_0],
\end{eqnarray*}
and one has the following proposition:
\begin{proposition}
\label{hybrid_rho2}
\begin{eqnarray}
\hat{\rho}_n^{[2]}  (r) =  \frac{2}{\pi} \int_0^\infty \frac{\mathrm{Tr}(\Gamma A_2) - \Re \left( \frac{\mathrm{Tr} (M^{-1} \Gamma A_2)}{\det(M)}  \right) }{s^2} \ ds \times \frac{1}{\pi^2 \det \Sigma_X},
\end{eqnarray}
where 
\begin{eqnarray*}
M = I-i \Gamma A_1.
\end{eqnarray*}
\end{proposition}
The proof is available in Appendix \ref{AppendixG}.

\subsection{On the repulsivity of charged zeros}
\label{sec:repuls_zeros}

In this section, we will see how different charges behave by computing the 2-point intensity between charges of the same sign as well as the charges of different signs. In what follows, we denote $\rho_n^{[2],++}$, $\rho_n^{[2],--}$ and $\rho_n^{[2],+-}$, the 2-point intensity between zeros charged positively, zeros charged negatively and zeros of different charges respectively.

In order to compute $\rho_n^{[2],++}$, $\rho_n^{[2],--}$ and $\rho_n^{[2],+-}$, we use the different 2-point intensities already computed in the previous sections, indeed one has these relationships:

\begin{eqnarray*}
\rho_n^{[2]} &=& \mathbb{E}[|J_0 J_r| ] = \rho_n^{[2],++} + \rho_n^{[2],--} + 2 \rho_n^{[2],+-},\\
\tilde \rho_n^{[2]} &=& \mathbb{E}[J_0 J_r ] = \rho_n^{[2],++} + \rho_n^{[2],--} - 2 \rho_n^{[2],+-},\\
\hat \rho_n^{[2]} &=& \mathbb{E}[|J_0| J_r ] = \rho_n^{[2],++} - \rho_n^{[2],--},
\end{eqnarray*}
therefore the repulsion between different charged zeros 
\begin{eqnarray*}
\rho_n^{[2],++} &=& \frac{1}{2} \ \left(\frac{\rho_n^{[2]} + \tilde \rho_n^{[2]}}{2} +   \hat \rho_n^{[2]}\right), \\
\rho_n^{[2],--} &=& \frac{1}{2} \ \left(\frac{\rho_n^{[2]} + \tilde \rho_n^{[2]}}{2} -   \hat \rho_n^{[2]}\right), \\
\rho_n^{[2],+-} &=&  \frac{\rho_n^{[2]} - \tilde \rho_n^{[2]}}{4},
\end{eqnarray*}
and define the normalized pair correlation function $g_n^{[2],++}$, $g_n^{[2],--}$ and $g_n^{[2],+-}$ as follows:
\begin{eqnarray*}
g_n^{[2],++} &=& \frac{\rho_n^{[2],++}}{\rho_n^{[1],+}\rho_n^{[1],+}}, \\
g_n^{[2],--} &=& \frac{\rho_n^{[2],--}}{\rho_n^{[1],-}\rho_n^{[1],-}}, \\
g_n^{[2],+-} &=& \frac{\rho_n^{[2],+-}}{\rho_n^{[1],+}\rho_n^{[1],-}}.
\end{eqnarray*}
In Fig. (\ref{Fig3}), $g_n^{[2],++}$, $g_n^{[2],--}$ and $g_n^{[2],+-}$ are plotted for $n=1$, $n=2$ and $n=3$.

At small scales, the zeros with the same charge represent a strong short-range repulsion.
As $r$ increases, this repulsion becomes progressively weaker, as $g_n^{[2],++}$ and $g_n^{[2],--}$ increases to 1.
For zeros with opposite charges, the short-range repulsion is weaker as $g_n^{[2],+-} $ remains larger near the origin.


\begin{figure*}
\begin{subfigure}{0.32\textwidth}
\centering
\includegraphics[scale = 0.4]{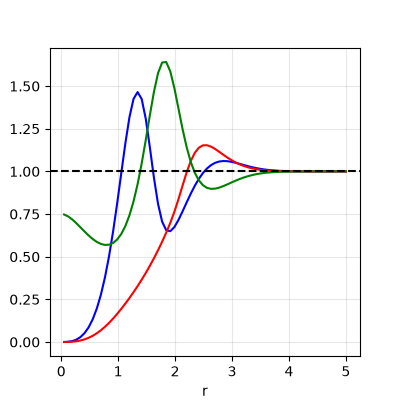}
\caption{}
\end{subfigure}
\hfill
\begin{subfigure}{0.32\textwidth}
\centering
\includegraphics[scale = 0.4]{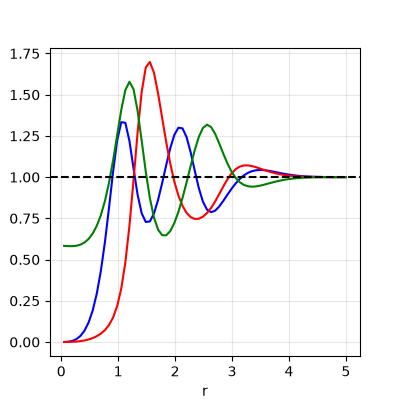}
\caption{}
\end{subfigure}
\hfill
\begin{subfigure}{0.32\textwidth}
\centering
\includegraphics[scale = 0.4]{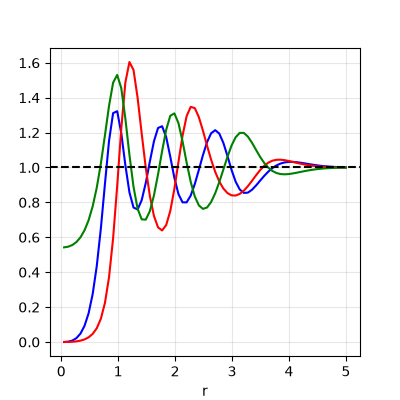}
\caption{}
\end{subfigure}
\caption{$g_n^{[2],++}$ (blue), $g_n^{[2],--}$ (red) and $g_n^{[2],+-}$ (green) (a) for $n=1$ (b) for $n=2$ and (c) for $n=3$.}
\label{Fig3}
\end{figure*}

\section{Detection of  non-degenerate zeros of $f_n$}\label{sec:detection}

In this section, we address the question of how to detect isolated zeros of a function $f_n$. For $n=0$, we already know that the zeros of $V_{h_0}$ correspond to the local minima of its spectrogram. For $n\ge 1$, however, not every local minimum corresponds to a zero.

To compute the zeros of $f_n$, for $n\ge 1$, one may first detect the sign change of both the real and imaginary parts simultaneously within a cell grid. However, this method can miss zeros that are not associated with such a sign change. Also, when the STFT presents high variations, the detection may then fail between consecutive points, which may also affects later the estimation of the Jacobian $\det J_{f_n}$ which depends on the derivatives of $\Re V_{h_n}$ and $\Im V_{h_n}$. 

To overcome these limitations, we look at the variation of the complex phase around a point \cite{SETE2024128595, haimi2026phasejumpsfastcomputationzeros}. Indeed, when a contour $\gamma$ encloses an isolated zero, its image under $f_n$ winds around the origin of the complex plane. This topological property is quantified by the winding number. 

\subsection{Detection based on winding number}
Since a zero is necessarily a local minima, one can first detect all the local minima of the spectrogram and use them as candidate points, and then determine if such a candidate is a zero by studying the phase variation of $f_n$ along a small contour surrounding it. 
First we define \cite{rudin2021principles}:
\begin{definition}
If $\gamma:[a,b] \to \C $ is a closed curve, the winding number of $\gamma$ around $z_0 \in \C$ is defined for $z_0 \notin \gamma ([a,b])$ as: 
\begin{equation}
    wind(\gamma, z_0) = \frac{1}{2\pi} \Delta_{t\in[a,b]} \arg(\gamma(t) - z_0).
\end{equation}
\end{definition}

So that the winding number measures how many times the curve winds around the point $z_0$ counted with orientation. In particular, if $\gamma$ is a positively oriented circle with center at $z_0 \in \C$ and a radius $r$ , ie $\gamma(t)=z_0+re^{it}, \ t \in [0, 2\pi]$, then \cite{rudin1974real}:
\begin{eqnarray}
wind(\gamma, z) = \begin{cases}
1, & |z-z_0|<r, \\
0, & |z-z_0|>r.
\end{cases}
\end{eqnarray}

Now consider $f_n$ and a closed contour $\gamma$ in the domain of $f_n$ such that $\forall t \in [a,b], \  f_n(\gamma(t)) \neq 0 $. Then $f_n \circ \gamma$ is a closed curve in $\C^* $, and the winding number of $f_n\circ \gamma$ around the point $0$, corresponds to: 
\begin{eqnarray*}
    wind(f_n\circ \gamma, 0) = \frac{1}{2\pi} \Delta_{t\in[a,b]} \arg(f_n(\gamma(t))),
\end{eqnarray*}

and therefore it measures the total phase variation of $f_n$ along $\gamma$. 

In practice, the winding number should be approximated by a finite number of samples along the contour. In our approach, these values are not necessarily located on the initial sampling grid. We first construct a separable bicubic interpolation of the real and imaginary part of $V_{h_n}$ \cite{jacquemet2018phase}, and evaluate it on equally separated points on the contour. 

Formally, if these points of the contour are denoted by $\zeta_0, \dots , \zeta_N $, the winding number can be approximated by the discrete sum:
\begin{eqnarray*}
wind(V_{h_n}\circ\gamma,0)
    \approx \frac{1}{2\pi}\sum_{k=0}^{N-1} \arg\!\bigl(\zeta_{k+1}\overline{\zeta_k}\bigr), \ 
    \zeta_{N}=\zeta_0.
\end{eqnarray*}

This approximation is based on the assumption that the phase variation between two consecutive samples remains sufficiently small because large phase variations can lead to incorrect estimate of the winding number. To avoid this situation, one can add a criterion on the phase variation proposed by \cite{haimi2026phasejumpsfastcomputationzeros}, which consists on keeping only points where:
\begin{eqnarray}
\max_{0\leq k\leq N-1} | \arg \ \bigl[ \zeta_{k+1} \overline{\zeta_k} \bigl] | < 0.9 \pi.
\end{eqnarray}

In our approach, the bicubic interpolation allows us to evaluate $V_{h_n}$ at an arbitrarily fine number of points along the contour and the angular sampling can be then adjusted to control the phase variations between two successive points.

The radius is then chosen such that the contour encloses only the candidate zero point while avoiding others candidate zeros. The choice of the radius is discussed in the next section.

\begin{algorithm}[H]
\caption{Detection of non-degenerate zeros of $f_n$}
\begin{algorithmic}[1]
\State \textbf{Input:} STFT $V_{h_n}\xi$ and a radius $R$
\State Compute the local minima of $|V_{h_n}|$
\State Interpolate $\Re V_{h_n}$ and $\Im V_{h_n}$ using bicubic splines to obtain $\tilde{V}_{h_n}$
\For{each local minimum $z_0$}
    \State Define $\gamma_R (\theta) = z_0 + e^{i\theta}$
    \State Set $\theta_k = \frac{2\pi k}{N}$ and $\zeta_k = \tilde{V}_{h_n}(\gamma(\theta_k))$
    \State Compute $wind (V_{h_n}\circ \gamma_R,0)
    = \frac{1}{2\pi}\sum_{k=0}^{N-1} \arg\!\bigl(\zeta_{k+1}\overline{\zeta_k}\bigr)$, $\zeta_{N}=\zeta_0$
    \If{$ | wind (V_{h_n}\circ \gamma_R,0) | =1$}
        \State Accept $z_0$ as zero and assign charge sign$(wind (V_{h_n}\circ \gamma_R,0))$.
    \EndIf
\EndFor
\end{algorithmic}
\end{algorithm}

Note that the winding number in our case can only take three values: $1$, $-1$ or $0$. To show that, let $z_0$ be an isolated zero of the polyanalytic function $f(z, \bar{z})$, then near $z_0$, we may write $$f(z,\bar z) \approx f(z_0,\bar z_0) + a (z - z_0) + b (\bar{z} - \bar{z}_0),$$
where $a = \frac{\partial f}{\partial z}(z_0,\bar z_0)$ and $b=\frac{\partial f}{\partial \bar{z}}(z_0,\bar z_0)$. $z_0$ is a zero so $f(z_0,\bar z_0)=0$ and on a small circle around $z_0$, $z=z_0 + r e^{i\theta}$ we get:
\begin{eqnarray*}
f(z,\bar z) \approx  a r e^{i\theta} + b r e^{-i\theta}  = r (a e^{i\theta} + b e^{-i\theta} )
\end{eqnarray*}
    
Writing $a =|a|e^{i\alpha}$ and $b =|b|e^{i\beta}$, one has:
\begin{eqnarray*}
f(z,\bar z) \approx r e^{i\frac{\alpha+\beta}{2}} \bigl( (|a| + |b|) \cos \psi  + i(|a| - |b|) \sin \psi \bigr)
\end{eqnarray*}
where $\psi = \theta + \frac{\alpha-\beta}{2}$.

This shows that the image of a circle is approximately a closed ellipse curve rotated by an angle of $\frac{\alpha+\beta}{2}$ in the complex plane. As $\theta$ runs from $0$ to $2 \pi$, so does $\psi$, and therefore $\cos \psi$ and $\sin \psi$ each complete one full period. Thus, in the nondegenerate case when $\abs a \neq \abs b$, and when the contour encloses the simple zero $z_0$, the point traces the ellipse exactly once and returns to its starting point and the winding number is necessarily ($+1$) or ($-1$), according to the sign of 
$$J_f(z_0)={\abs a}^2 - {\abs b}^2.$$
If the contour does not enclose any zero, then the winding number is ($0$).

\subsection{The optimal radius of the contour of the winding numbers for zeros detection }
Zeros detection using the proposed algorithm depends strongly on the chosen radius of the contour, indeed a small circle around the zero will detect the corresponding zero, however, as the radius of the circle increases, multiple zeros will fall into the same contour and the algorithm may fail. To determine the optimal radius of the corresponding contour, we will use densities as a criterion. 

Let $\overline{\rho^{[1],+}_{n}}$ and $\overline{\rho^{[1],-}_{n}}$ be the estimated average of the positive and negative first-point intensities respectively over a number of realizations and define the following error:

\begin{eqnarray}
E_n (R) = \frac{1}{2} \ ( \varepsilon_{n,+} + \varepsilon_{n,-}),
\end{eqnarray}
where
\begin{eqnarray}
\varepsilon_{n,+}=\frac{\left|\ \overline{\rho^{[1],+}_{n}} - \rho^{[1],+}_n \right|}{\rho^{[1],+}_n}, \ 
\varepsilon_{n,-}=\frac{\left|\ \overline{\rho^{[1],-}_{n}} - \rho^{[1],-}_n \right|}{\rho^{[1],-}_n}.
\end{eqnarray}

The optimal radius is then obtained as:
\begin{eqnarray*}
R_n^*=\argmin_{R>0}\ E_n(R).
\end{eqnarray*}
We plot the error as a function of the radius $R$ in Fig \ref{Fig4}.

For small values of the radius, the contour is very close to the zero candidate, and the evaluation of the winding number becomes more sensitive to numerical interpolation. As $R$ increases, the estimation becomes more stable leading to a decrease in the error until it reaches a minimum around $R_n^*$. 

However, when $R > R_n^*$, the circle may enclose additional zeros. Since the winding number is additive with respect to the topological charges, its value is then no necessarily equal to $\pm 1$. For example, two zeros with the same charge give a winding number of modulus 2, whereas a zero with charge $+1$ together with a zero with charge $-1$ gives a winding number of 0.

Thus, the detection criterion $| wind (V_{h_n}\circ \gamma_R,0) | =1$, which is designed to a contour containing only one simple zero, may no longer be satisfied when multiple zeros are enclosed, leading to an increase in the detection error for large values of $R$.

\begin{figure*}
\begin{subfigure}{0.32\textwidth}
\centering
\includegraphics[scale = 0.4]{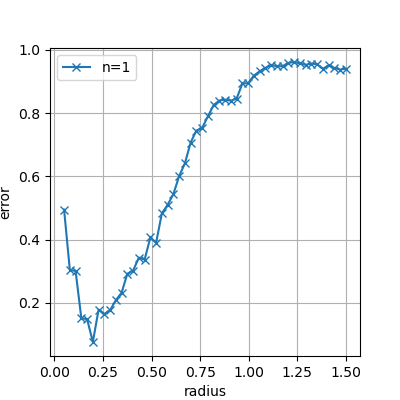}
\caption{}
\end{subfigure}
\hfill
\begin{subfigure}{0.32\textwidth}
\centering
\includegraphics[scale = 0.4]{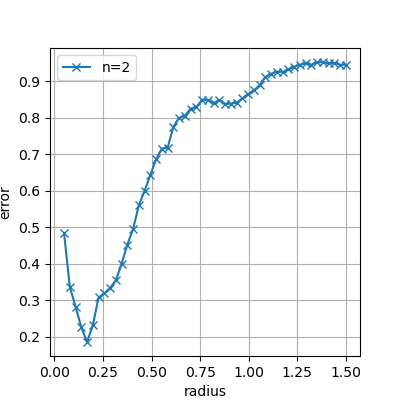}
\caption{}
\end{subfigure}
\hfill
\begin{subfigure}{0.32\textwidth}
\centering
\includegraphics[scale = 0.4]{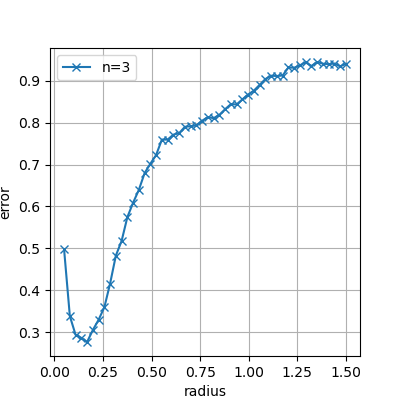}
\caption{}
\end{subfigure}
\caption{ The error $E_n$ as a function of the radius $r$ (a): for $n=1$ (b): for $n=2$ (c): for $n=3$}
\label{Fig4}
\end{figure*}


\section{Conclusion}

In this work, we studied the zeros of the spectrogram of complex white noise analyzed with Hermite windows of order $n\geq 1$, and showed that these zero patterns are statistically organized into two distinct populations, carrying either a positive or a negative charge.  We derived closed-form expressions for the one-point intensity and its decomposition by charge, as well as for the unsigned, signed, and hybrid two-point intensities. In particular, we showed that the pair correlation function $r\mapsto g_n^{[2]}(r)$ converges as $r\to 0$ to a strictly positive constant $l_n$, in contrast with the total repulsion phenomenon observed for the Gaussian window. Finally, we proposed a practical method for detecting the zeros of a spectrogram based on the winding number computed on small circular contours around candidate local minima. As a perspective, we plan to exploit the structure of these zeros to discriminate signal and noise using Delaunay triangulation of 
zeros according to whether they are charged positively or negatively.

\section{Appendix}
\subsection{Proof of Proposition \ref{prop:Hkn}}
\label{AppendixA}
Let us first remark that 
\begin{eqnarray}
\label{isometry-B}
\begin{aligned}
V_{h_n}h_k(u,-v) = \langle h_k, M_{-v} T_u h_n \rangle =  \langle M_{v} T_u h_n,h_k \rangle = \langle   {\cal B} M_v T_u h_n,  {\cal B} h_k \rangle_{\F}.
\end{aligned}
\end{eqnarray}
Putting $z_0 = u+iv$, let us first remark that
$$
({\cal B}  M_v T_{u} h_n)(z):= e^{- \frac{\pi}{2}|z_0|^2 + i\pi v u} \frac{\pi^\frac{n}{2}}{\sqrt{n!}} F(z),
$$
where 
\begin{equation}
\label{def:F}
F(z)=e^{\pi z z_0} (z - \bar{z_0})^n.
\end{equation}
To compute \eqref{isometry-B}, using the fact that ${\cal B}h_k (z) = \sqrt\frac{\pi^k}{k!} z^k$ 
is an orthonormal basis in $\F$, we identify the coefficient of $z^k$ in $F(z)$, denoted by $[z^k]F(z)$, and then multiply it by $\sqrt\frac{k!}{\pi^k}$.
Writing the binomial expansion of $(z - \bar{z}_0)^n$ and the Taylor series
of $e^{\pi z z_0}$:
$$
(z - \bar{z}_0)^n = \sum_{i=0}^{n} \binom{n}{i} (-\bar{z}_0)^{n-i}\, z^i,
\
e^{\pi z z_0} = \sum_{p=0}^{\infty} \frac{(\pi z_0)^p}{p!}\, z^p.
$$
Hence
\[
F(z) = 
\sum_{i=0}^{n} \binom{n}{i}(-\bar{z}_0)^{n-i}
\sum_{p=0}^{\infty} \frac{(\pi z_0)^p}{p!}\, z^{i+p}.
\]
The coefficient of $z^k$ is obtained by setting $p = k - i$, which requires
$0 \leq i \leq \min(n,k)$:
\[
[z^k]\, F(z)
= 
\sum_{i=0}^{\min(n,k)} \binom{n}{i}
\frac{(-\bar{z}_0)^{n-i}(\pi z_0)^{k-i}}{(k-i)!}.
\]
Re-indexing $j=k-i$, then the sum goes from $\max(0, k-n)$ to $k$, one gets: 
$$
[z^k]\, F(z)
= 
\sum_{j=\max(0,k-n)}^{k} \binom{n}{k-j}
\frac{(-\bar{z}_0)^{n-k+j}(\pi z_0)^j}{j!}.
$$
If $k\leq n$, 
\begin{eqnarray*}
[z^k]\, F(z)
&=& (-\bar{z}_0)^{n-k}\sum\limits_{j=0}^{k} \binom{n}{k-j}
\frac{(-\pi |z_0|^2)^j}{j!},
\end{eqnarray*}
which corresponds to the standard form of the generalized Laguerre polynomial
\[
L_n^{(\alpha)}(x)
= \sum_{j=0}^{n} \binom{n+\alpha}{n-j} \frac{(-x)^j}{j!}.
\]
So for $k< n$, 
\begin{eqnarray*}
    \langle   {\cal B} M_v T_u h_n,  {\cal B} h_k \rangle_{\F} = 
    e^{ - \frac{\pi}{2}|z_0|^2 + i\pi v u }\, \sqrt{\frac{\pi^{n-k} k!}{n!}} 
    (-\bar{z}_0)^{n-k} L_k^{(n-k)}(\pi |z_0|^2).
\end{eqnarray*}
Now if $k>n$,
\[
[z^k]\, F(z)
=\sum\limits_{j=k-n}^{k} \binom{n}{k-j}
\frac{(-\bar{z}_0)^{n-k+j}(\pi z_0)^j}{j!}. \\
\]
Re-indexing this time $m=j-k+n$, one gets: 
\[
[z^k]\, F(z)
={z_0}^{k-n} \pi^{k-n}  \sum\limits_{m=0}^{n} \binom{n}{m}
\frac{(-\pi |z_0|^2)^{m} m!}{m! (m+k-n)!}. \\
\]
As $\binom{n}{m} \frac{m!}{(m+k-n)!} =\frac{n!}{k!} \binom{k}{n-m}$, one obtains:
\begin{align*}
[z^k]\, F(z)
&={z_0}^{k-n} \frac{\pi^k}{\pi^n} \frac{n!}{k!}  \sum_{m=0}^{n} \binom{k}{n-m}
\frac{(-\pi |z_0|^2)^m}{m!}.
\end{align*}
Thus 
\begin{eqnarray*}
    \langle   {\cal B} M_v T_u h_n,  {\cal B} h_k \rangle_{\F} =
 e^{ - \frac{\pi}{2}|z_0|^2 + i\pi v u }\, 
 \sqrt{\frac{\pi^{k-n} n!}{k!}} z_0^{k-n} L_n^{(k-n)}(\pi |z_0|^2).
\end{eqnarray*}

\subsection{Proof of Proposition \ref{unsigned_rho2}}
\label{AppendixB}
For every $s,t \in \R$, using we recall that
\begin{equation}
\varphi(s,t) =
\mathbb{E}\ \left[
\exp\bigl(i(sJ_0+tJ_r)\bigr)
\right].
\end{equation}
Using the Dirichlet integral 
$$\int_0^{\infty} \frac{\sin (x)}{x} \ dx=\frac{\pi}{2},$$  
one can prove that for all $x \in \R$:
\begin{equation*}
    |x| = \frac{2}{\pi}\int_0^{\infty} \frac{1-\cos(sx)}{s^2} \ ds.
\end{equation*}
Applying this to both $J_0$ and $J_r$,
we obtain:
\begin{eqnarray}
\label{integral_J0}
|J_0| = \frac{2}{\pi}\int_0^{\infty} \frac{1-\cos(s J_0)}{s^2} \ ds, \ 
|J_r| = \frac{2}{\pi}\int_0^{\infty} \frac{1-\cos(s J_r)}{s^2} \ ds,
\end{eqnarray}
therefore
\begin{eqnarray*}
|J_0 J_r| &=& 
\frac{4}{\pi^2}\iint_{(\R^+)^2} \frac{\bigr( 1-\cos(s J_0)\bigl) \bigr( 1-\cos(t J_r)\bigl) }{s^2 t^2} \ ds \ dt \\
&=&\frac{4}{\pi^2}\iint_{(\R^+)^2} \frac{ 1-\cos(s J_0)-\cos(t J_r) + \cos(s  J_0)  \cos(t  J_r) }{s^2 t^2} \ ds \ dt. 
\end{eqnarray*}
Then using
\begin{eqnarray*}
\cos(s  J_0)  \cos(t  J_r) = \frac{1}{2} \bigr( \cos(sJ_0 + tJ_r) + \cos(sJ_0 - tJ_r) \bigl),
\end{eqnarray*}
and
\begin{eqnarray*}
\mathbb{E}\ \left[
\cos\bigl(sJ_0+t J_r\bigr) \right] = \Re \varphi(s,t),
\end{eqnarray*}
one obtains
\begin{equation*}
\mathbb{E} \ \left[ \left| J_0 J_r \right| \right] = 
\frac{4}{\pi^2}
\int_0^\infty \int_0^\infty
\frac{\psi (s,t)}{s^2t^2}  ds dt,
\end{equation*}
where 
\begin{equation*}
    \psi(s,t)= 1-\Re\varphi(s,0)-\Re\varphi(0,t)
+\frac{1}{2}\bigl(\Re\varphi(s,t)+\Re\varphi(s,-t)\bigr),
\end{equation*}
which ends the proof.

\subsection{Implementation details}
\label{AppendixC}
The integral 
\begin{eqnarray}
\label{defJ}
J:= \int_0^\infty \int_0^\infty
\frac{\psi (s,t)}{s^2t^2}  ds dt,
\end{eqnarray}
is theoretically correct, but in practice becomes rapidly numerically unstable when $r$ increases. 
In order to explain this, let us consider the integral $$\frac{\pi}{2}|x|=  I(x)=\int_0^\infty \frac{1-\cos(sx)}{s^2}  ds. $$
When $|x|$ is too large, $s \rightarrow \cos(sx)$ varies rapidly, indeed, its zeros and the distance between two consecutive zeros are given respectively by:
\begin{equation*}
s_{k}=\frac{(2k+1)\pi}{2|x|}
\textrm{ and }
s_{k+1}-s_k= \frac{\pi}{|x|}.
\end{equation*}
Thus the variations of $\cos(sx)$ are located in intervals of length proportional to $\frac{1}{|x|}$ and a numerical discretization must see such variations which becomes numerically difficult especially in our case when $|x|$ is large.  

To overcome this problem, we make the change of variable $u=|x|s$ and the integral becomes 
$$ I(x)=|x| \int_0^\infty \frac{1-\cos(u)}{u^2}  du, $$
and the variations inside the integral no longer depend on $x$. In our case, the changes of variable that should be done are 
$u=|J_0| s = \sqrt{J_0^2} s$ and $v=|J_r| t=\sqrt{J_r^2}  t$, but since $J_0$ and $J_r$ are random realizations, we consider instead
\begin{eqnarray*}
u = \sqrt{\mathbb{E}[J_0^2]} s, \ 
v = \sqrt{\mathbb{E}[J_r^2]} t,
\end{eqnarray*}
and the integral \eqref{defJ} becomes 
\begin{eqnarray*}
J = a_0 a_r\int_0^\infty \int_0^\infty
\frac{\psi (\frac{u}{a_0},\frac{v}{a_r})}{u^2v^2}  du dv.
\end{eqnarray*}
To compute the quantities $a_0$ and $a_r$, we use the following result:
\begin{proposition}
\label{trace_gamma_A}
\begin{figure*}
\centering
\begin{tabular}{@{}cc@{}}
\includegraphics[scale = 0.49]{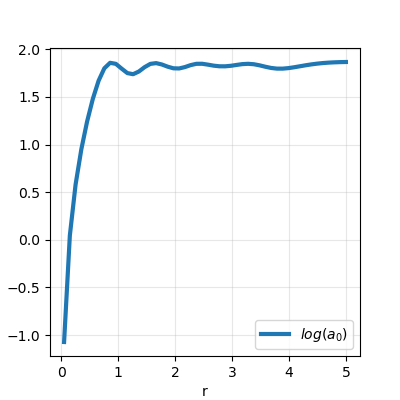}&\includegraphics[scale = 0.49]{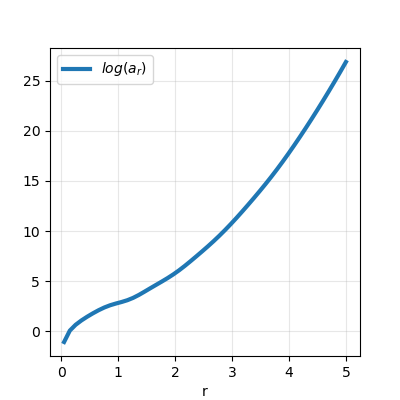}\\
(a)&(b)
\end{tabular}
\caption{Logarithmic plots of (a) $a_0$ and (b) $a_r$ as function of $r$.}
\label{Fig5}
\end{figure*}
Let us consider a Gaussian vector $Z$ with covariance matrix $\Gamma $ and a Hermitian matrix $A$, one has: 
\begin{eqnarray}
\mathbb{E}[(Z^* A Z )^2] &=& (\mathrm{Tr}(\Gamma A))^2 + \mathrm{Tr}(\Gamma A \Gamma A)
\end{eqnarray}
\end{proposition}
\begin{proof}

We have 
\begin{eqnarray*}
(Z^* A Z)^2 = \left( \sum_{i,j} \bar{z_i} A_{i,j}z_j\right)^2 = \sum_{i,j}\sum_{k,l} \bar{z_i} A_{i,j}z_j \bar{z_k} A_{k,l}z_l,
\end{eqnarray*}
then introducing the expectation
\begin{eqnarray*}
\mathbb{E}[(Z^* A Z)^2] = \sum_{i,j}\sum_{k,l}  A_{i,j}A_{k,l}\mathbb{E}[\bar{z_i} z_j \bar{z_k} z_l],
\end{eqnarray*}
and using Isserlis's theorem \cite{isserlis1918formula}
\begin{eqnarray*}
\mathbb{E}[\bar{z_i} z_j \bar{z_k} z_l] = \Gamma_{i,j}\Gamma_{k,l} + \Gamma_{i,l}\Gamma_{k,j},
\end{eqnarray*}
we obtain
\begin{eqnarray*}
\mathbb{E}[(Z^* A Z)^2] &=& \sum_{i,j}\sum_{k,l}  A_{i,j}A_{k,l}\Gamma_{i,j}\Gamma_{k,l} + \sum_{i,j}\sum_{k,l}  A_{i,j}A_{k,l}\Gamma_{i,l}\Gamma_{k,j} \\
&=& \sum_{i,j}A_{i,j}\Gamma_{i,j}\sum_{k,l}  A_{k,l}\Gamma_{k,l} + \sum_{i,j}\sum_{k,l}  A_{i,j}A_{k,l}\Gamma_{i,l}\Gamma_{k,j} \\
&=& \left(\sum_{i,j}A_{i,j}\Gamma_{i,j}\right)^2 + \sum_{i,j}\sum_{k,l}  A_{i,j}A_{k,l}\Gamma_{i,l}\Gamma_{k,j}.
\end{eqnarray*}
On the other hand, and using the symmetry of $\Gamma$
\begin{eqnarray*}
\mathrm{Tr}(A\Gamma)=\sum_i (A\Gamma)_{i,i}=\sum_{i,j} A_{i,j}\Gamma_{j,i}=\sum_{i,j} A_{i,j}\Gamma_{i,j},
\end{eqnarray*}
and
\begin{eqnarray*}
\mathrm{Tr}(A\Gamma A\Gamma)&=&\sum_{i,k}(A\Gamma)_{i,k}(A\Gamma)_{k,i} 
= \sum_{i,k} \sum_j A_{i,j}\Gamma_{j,k} \sum_l A_{k,l}\Gamma_{l,i} \\
&=& \sum_{i,j, k, l} A_{i,j}\Gamma_{j,k}A_{k,l}\Gamma_{l,i} = \sum_{i,j, k, l} A_{i,j}A_{k,l}\Gamma_{j,k}\Gamma_{l,i}.
\end{eqnarray*}
Therefore
\begin{eqnarray*}
\mathbb{E}[(Z^* A Z)^2] = \left( \mathrm{Tr}(A\Gamma) \right)^2 + \mathrm{Tr}(A\Gamma A\Gamma)
= \left( \mathrm{Tr}(\Gamma A) \right)^2 + \mathrm{Tr}(\Gamma A\Gamma A).
\end{eqnarray*}
\end{proof}
The two quantities $a_0$ and $a_r$ are then given by:
\begin{eqnarray*}
a_0^2=\mathbb{E}[J_0^2] = (\mathrm{Tr}(\Gamma A_1))^2 + \mathrm{Tr}(\Gamma A_1 \Gamma A_1), \\ 
a_r^2 =\mathbb{E}[J_r^2] = (\mathrm{Tr}(\Gamma A_2))^2 + \mathrm{Tr}(\Gamma A_2 \Gamma A_2).
\end{eqnarray*}
In Fig. \ref{Fig5}, the logarithms of $a_0$ and $a_r$ are 
plotted, motivating the changes of variable explained above.

\subsection{Proof of the monotonicity of $l_n$}
\label{AppendixD}
Define the sequence $(x_n)_{n\ge 0}$ and the function $\phi$ respectively by
\begin{eqnarray*}
x_n=\frac{n(n+1)}{(2n+1)^2}
\textrm{ and } 
\phi(n)=n(n+1).
\end{eqnarray*}
One then remarks that $\phi(n)$ is strictly increasing and that:
\begin{eqnarray*}
l_n = \frac{n(n+1)(n^2+n+1)}{(2n^2+2n+1)^2} = x_{\phi(n)}.
\end{eqnarray*}
So to prove that $l_n$ is increasing it is sufficient to prove that
$x_n$ is increasing, thus we compute
\begin{eqnarray*}
x_{n+1}-x_n &= \frac{(n+1)(n+2)}{(2n+3)^2} - \frac{n(n+1)}{(2n+1)^2}= (n+1) \left( \frac{n+2}{(2n+3)^2} - \frac{n}{(2n+1)^2}\right)\\
&= (n+1) \left( \frac{(n+2)(2n+1)^2 - n(2n+3)^2}{(2n+3)^2(2n+1)^2}\right)=\frac{2(n+1)}{(2n+3)^2(2n+1)^2} >0,
\end{eqnarray*}
which ends the proof.

\subsection{Proof of Proposition \ref{limit_pair_function}}
\label{AppendixE}
In order to compute \[ \lim_{r\to 0} g_n^{[2]}(r), \]  we need to compute the asymptotics, when $r$ is small, of: 
\begin{eqnarray*}
\det \Sigma_X(r) \textrm{ and } 
\mathbb{E}[\left| J_0 J_r \right|].
\end{eqnarray*}
Let us consider $n\ge0$ and put $x=r^2$.
We have the following Taylor expansions: 
\begin{eqnarray*}
L^{[0]} = L_n(x) = 1-nx+\frac{n(n-1)}{4}x^2 + O(x^4), \
e^x = 1+x+\frac{x^2}{2} + O(x^3), \\
\end{eqnarray*}
and thus
\begin{eqnarray}
\label{eq:detX}
\Delta = \det \Sigma_X(r)=e^x - (L^{[0]})^2 = (2n+1)x+\frac{1+n-3n^2}{2}x^2 + O(x^3) .
\end{eqnarray}

Then, to compute $\mathbb{E}[\left| J_0 J_r \right|]$ when 
$r$ is small we first compute first order in $x$ of $\Gamma$. 
For that purpose we first compute 
\begin{eqnarray*}
L^{[1]} = L_n'(x) = -n+\frac{n(n-1)}{2}x + O(x^2), \
L^{[2]} = L_n''(x) = \frac{n(n-1)}{2} + O(x), \\
B^{[0]} = L^{[0]}-L^{[1]} = n+1 - \frac{n(n+1)}{2}x +O(x^2), \\
B^{[1]} = L^{[1]}-L^{[2]} = -\frac{n(n+1)}{2} + O(x). \\
\end{eqnarray*}
Recalling that 
\begin{eqnarray*}
\Gamma &= \sum\nolimits_Y - (\sum\nolimits_0)^H (\sum\nolimits_X)^{-1} \sum\nolimits_0,\\
\end{eqnarray*}
we denote the columns of $\Sigma_0$ defined in \eqref{def:Sigma_0} 
by $c_1, c_2, c_3, c_4$, namely
\begin{eqnarray*}
c_1=
\begin{pmatrix}
0\\
rB^{[0]}
\end{pmatrix}, \ 
c_2=
\begin{pmatrix}
0\\
-rL^{[1]}
\end{pmatrix}, \
c_3=
\begin{pmatrix}
rL^{[1]}\\
r e^{r^2}
\end{pmatrix},
c_4=
\begin{pmatrix}
r L^{[1]}\\
0
\end{pmatrix},
\end{eqnarray*}
and $\Sigma_X^{-1}$ reads: 
$$
\Sigma_X^{-1} = \frac{1}{\Delta} 
\begin{pmatrix}
e^x &-L^{[0]}\\
-L^{[0]}& 1
\end{pmatrix}.
$$
Considering the definition \eqref{def:Sigma_Y} of 
$\Sigma_Y$, we compute the expansion of $\Gamma_{11}$:
\begin{eqnarray*}
\Gamma_{11}=(\Sigma_Y)_{11} - c_1^T \Sigma_X^{-1} c_1 
           = n+1 - \frac{x (B^{[0]})^2}{\Delta}. 
\end{eqnarray*}
Using the expression of $\Delta$ introduced above and since
\begin{eqnarray*}
x(B^{[0]})^2 = (n+1)^2x-n(n+1)^2x^2 + O(x^3),
\end{eqnarray*}
we get, after some simple computations, that 
\begin{eqnarray*}
c_1^T \Gamma_X^{-1} c_1 
&=\frac{(n+1)^2}{(2n+1)} - \frac{(n+1)^2(n^2+3n+1)}{2(2n+1)^2}x + O(x^2).
\end{eqnarray*}
Therefore:
\begin{eqnarray*}
\Gamma_{11} = \frac{n(n+1)}{(2n+1)} + \frac{(n+1)^2(n^2+3n+1)}{2(2n+1)^2}r^2 + O(r^4).
\end{eqnarray*}
Similarly, we may write: 
\begin{eqnarray*}
\Gamma_{12}&=&(\Sigma_Y)_{12} - c_1^T \Sigma_X^{-1} c_2 = \frac{x B^{[0]} L^{[1]}}{\Delta} = -\frac{n(n+1)}{(2n+1)} + \frac{n^2(n+1)^2}{2(2n+1)^2}r^2 + O(r^4)\\
\Gamma_{22}&=&(\Sigma_Y)_{22} - c_2^T \Sigma_X^{-1} c_2 
= n - \frac{x (L^{[1]})^2}{\Delta} = \frac{n(n+1)}{(2n+1)} + \frac{n^2(n^2-n-1)}{2(2n+1)^2}r^2 + O(r^4)\\
\Gamma_{13}&=&(\Sigma_Y)_{13} - c_1^T \Sigma_X^{-1} c_3 = 
n+1-n(n+1)x+O(x^2) - \frac{x B^{[0]} (e^x-L^{[1]}L^{[0]})}{\Delta} \\
 &= &\frac{n(n+1)}{(2n+1)} - \frac{(n+1)^2(3n^2+3n+1)}{2(2n+1)^2}r^2 + O(r^4)\\
\Gamma_{23}&=&(\Sigma_Y)_{23} - c_2^T \Sigma_X^{-1} c_3 =-\frac{n(n+1)}{2}x+O(x^2) +\frac{xL^{[1]}(e^x - L^{[0]}L^{[1]})}{\Delta} \\
 &=& -\frac{n(n+1)}{(2n+1)} + \frac{n(n+1)(n^2-n-1)}{2(2n+1)^2}r^2 + O(r^4)\\
\Gamma_{33}&=&(\Sigma_Y)_{33} - c_3^T \Sigma_X^{-1} c_3 \\
&=&n+1+(n+2)x+O(x^2) - \frac{x (e^x(L^{[1]})^2 - 2 L^{[0]} L^{[1]} e^x +e^{2x})}{\Delta} \\
\Gamma_{14}&=&(\Sigma_Y)_{14} - c_1^T \Sigma_X^{-1} c_4 
= -\frac{n(n+1)}{2}x+O(x^2)+\frac{x B^{[0]} L^{[1]}L^{[0]}}{\Delta} \\
&=& -\frac{n(n+1)}{(2n+1)} + \frac{n(n+1)(n^2-n-1)}{2(2n+1)^2}r^2 + O(r^4)\\
\Gamma_{24}&=&(\Sigma_Y)_{24} - c_2^T \Sigma_X^{-1} c_4 = 
n-n(n-1)x+O(x^2) + \frac{x(L^{[1]})^2L^{[0]}}{\Delta} \\
 &=& \frac{n(n+1)}{(2n+1)} - \frac{n(3n+2)(n^2-n-1)}{2(2n+1)^2}r^2 + O(r^4)\\
 \Gamma_{34}&=&(\Sigma_Y)_{34} - c_3^T \Sigma_X^{-1} c_4 
=- \frac{x e^x L^{[1]} (L^{[1]}-L^{[0]})}{\Delta} \\
&=& -\frac{n(n+1)}{(2n+1)} + \frac{n(n+1)(n^2-3n-2)}{2(2n+1)^2}r^2 + O(r^4)\\
\Gamma_{44}&=&(\Sigma_Y)_{44} - c_4^T \Sigma_X^{-1} c_4 = 
 n+nx+O(x^2) -\frac{x e^x (L^{[1]})^2}{\Delta} \\
    &=& \frac{n(n+1)}{(2n+1)} + \frac{n(n^3+3n^2+5n+2)}{2(2n+1)^2}r^2 + O(r^4)
\end{eqnarray*}

Now let introduce the quantities 
\begin{eqnarray*}
D_0=z_1-z_2, \
S_0=z_1+z_2, \
D_1=z_3-z_4,\
S_1=z_3+z_4.
\end{eqnarray*}
One first remark that 
\begin{eqnarray*}
J_0 = \Re(\overline{D_0}S_0), \ 
J_r = \Re(\overline{D_1}S_1)
\end{eqnarray*}
and thus $\mathbb{E} [ |J_0 J_r|] = 
\mathbb{E} [|\Re(\overline{D_0}S_0) \Re(\overline{D_1}S_1)|] $.
Let us first study $D_0$ and $D_1$, for which we have: 
\begin{eqnarray*}
{\mathbb E}[|D_0|^2] = \Gamma_{11} + \Gamma_{22}-2\Gamma_{12} = \frac{4n(n+1)}{2n+1} +O(r^2),\\
{\mathbb E}[|D_1|^2] = \Gamma_{33} + \Gamma_{44}-2\Gamma_{34}  = 
\frac{4n(n+1)}{2n+1} +O(r^2),\\
{\mathbb E}[D_0 \overline{D_1}] = \Gamma_{13} + \Gamma_{24} - \Gamma_{14}-\Gamma_{23}=\frac{4n(n+1)}{2n+1} +O(r^2). 
\end{eqnarray*}
From this we deduce that 
$$
{\mathbb E}[|D_1-D_0|^2] \underset{r \rightarrow 0}{\longrightarrow}  0.
$$
Furthermore as $Z$ is a Gaussian vector that converges in distribution to ${\cal C} {\cal N} (0,\Gamma_0)$, where  $\Gamma_0$ is the limit of $\Gamma$ when $r$ tends to $0$. 
$D_0$ converges in distribution to ${\cal C} {\cal N} (0,A^T\Gamma_0A)$, with $A = (1,-1,0,0)^T$. 
Similarly $D_1$ converges in distribution to ${\cal C} {\cal N} (0,B^T\Gamma_0B)$, with  $B = (0,0,1,-1)^T$.
Now, we can prove that $D_0$ and $D_1$ converge to the same distribution. Indeed, let us recall for consistency:
\begin{theorem}[Slutsky's theorem, \cite{casella2024statistical} p.239] 
Let $X_n$ and $Y_n$ be sequences of random elements. If $X_n$ converges in distribution to a random variable $X$ and $Y_n$ converges in probability to a constant $c$, then:
\begin{eqnarray*}
X_n + Y_n \overset{d}{\longrightarrow} X+c.
\end{eqnarray*}
\end{theorem}
In our case, assume $D_0 \overset{d}{\rightarrow} D$ and 
$D_1 \overset{d}{\rightarrow} \tilde D$ 
then we also have  
$$
D_1 = D_0 + D_1-D_0  \overset{d}{\longrightarrow} D,
$$
since $D_0 \overset{d}{\longrightarrow} D$, $D_1-D_0\overset{L^2}{\longrightarrow} 0$ and thus also in probability, and finally applying Slutsky's theorem. Let us call 
$D$ the common limit of $D_0$ and $D_1$. 

Now let us consider the variables $S_0$ and $S_1$, 
for which we have:  
\begin{eqnarray*}
{\mathbb E}[|S_0|^2] = \Gamma_{11} + \Gamma_{22}+2\Gamma_{12}=\frac{n^2+n+1}{2} r^2 +O(r^4), \\
{\mathbb E}[|S_1|^2] = \Gamma_{33} + \Gamma_{44}+2\Gamma_{34} 
= \frac{n^2+n+1}{2} r^2 +O(r^4),\\
{\mathbb E}[S_0 \overline{S_1}] = \Gamma_{13} + \Gamma_{14}+\Gamma_{23}+\Gamma_{24}=-\frac{n^2+n+1}{2} r^2 + O(r^4). \end{eqnarray*}

Now let us define the variables $\tilde{S_0}=\frac{S_0}{r}$ and $\tilde{S_1}=\frac{S_1}{r}$, which are Gaussian variable 
converging in distribution to $\tilde S$ and $\tilde{\tilde S}$ respectively. From the equalities just above we have
\begin{eqnarray*}
{\mathbb E}[|\tilde{S_0} + \tilde{S_1}|^2] = {\mathbb E}[|\tilde{S_0}|^2] + {\mathbb E}[|\tilde{S_1}|^2]+ 2 \Re {\mathbb E} [\tilde{S_0} \overline{\tilde{S_1}}] \underset{r\rightarrow 0}{\longrightarrow} 0.
\end{eqnarray*}
Using the same reasoning before, we write: 
$$
\tilde{S_1} = -\tilde{S_0} + \tilde{S_1}+ \tilde{S_0} 
\overset{d}{\longrightarrow} - \tilde S.
$$
Now let us compute
\begin{eqnarray*}
{\mathbb E}[D_0 \overline{S_0}] =  \Gamma_{11} - \Gamma_{22}=O(r^2), \ 
{\mathbb E}[D_r \overline{S_r}] =  \Gamma_{33} - \Gamma_{44}=O(r^2),
\end{eqnarray*}
which means that the variables $D_0$ and $\tilde{S_0}$ are asymptotically independent. Similarly, $D_1$ and $\tilde{S_1}$ are asymptotically independent. Therefore we can write
\begin{eqnarray*}
\frac{J_0}{r} \overset{d}{\longrightarrow} \Re(\overline{D}\tilde{S}), \
\frac{J_r}{r} \overset{d}{\longrightarrow} -\Re(\overline{D}\tilde{S}),
\end{eqnarray*}
therefore
\begin{eqnarray*}
\frac{|J_0 J_r|}{r^2} \overset{d}{\longrightarrow}
\left[\Re (\overline{D}\tilde{S}) \right]^2.
\end{eqnarray*}

Putting $D = D_x + i D_y$ and $\tilde{S} = \tilde S_x+i 
\tilde S_y$
then
\begin{eqnarray*}
\left [\Re (\overline{D}\tilde{S}) \right]^2 = 
D_x^2\tilde S_x^2+D_y^2\tilde S_y^2-2D_xD_y\tilde S_x \tilde S_y.
\end{eqnarray*}
Since the variables $D$ and $\tilde{S}$ are circular symmetric Gaussian variables, then:
\begin{eqnarray*}
{\mathbb E}[D_xD_y] = {\mathbb E}[\tilde S_x \tilde S_y] = 0,
\end{eqnarray*}
and 
\begin{eqnarray*}
{\mathbb E}[D_x^2] = {\mathbb E}[D_y^2] = \frac12  {\mathbb E}[|D|^2], \ 
{\mathbb E}[\tilde S_x^2] = {\mathbb E}[\tilde S_y^2] = \frac12  {\mathbb E}[|\tilde{S}|^2].
\end{eqnarray*}
Using Isserlis's theorem, and asymptotic independence, we get 
\begin{eqnarray*}
{\mathbb E}[D_x^2\tilde S_x^2] &=& {\mathbb E}[D_x^2]{\mathbb E}[\tilde S_x^2],\\ 
{\mathbb E}[D_y^2\tilde S_y^2] &=& {\mathbb E}[D_y^2]{\mathbb E}[\tilde S_y^2],\\
{\mathbb E}[D_x D_y \tilde S_x \tilde S_y ] &=& 0.
\end{eqnarray*}
Combining all these results, one obtains :
\begin{eqnarray*}
\frac{\mathbb{E}[\left| J_0 J_r \right|]}{r^2} \underset{r\rightarrow 0}{\longrightarrow}
\frac{1}{2} {\mathbb E}[|D|^2]{\mathbb E}[|S|^2]= 
\frac{n(n+1)(n^2+n+1)}{(2n+1)},
\end{eqnarray*}
and using (\ref{eq:detX}), 
\begin{eqnarray*}
\frac{\det \Sigma_X}{r^2}\underset{r\rightarrow 0}{\longrightarrow} (2n+1),
\end{eqnarray*}
one gets that
\begin{eqnarray*}
\rho_n^{[2]}(r)=\frac{\mathbb{E}[\left| J_0 J_r \right|]}{\pi^2 \det \Sigma_X} \underset{r\rightarrow 0}{\longrightarrow} \frac{n(n+1)(n^2+n+1)}{(2n+1)^2 \pi^2}.
\end{eqnarray*}
Also we know that (\ref{rho1}):
\begin{eqnarray*}
\rho_n^{[1]}=\frac{1}{\pi} (n+1/2+\frac{1}{4n+2})
=\frac{1}{\pi} \ \frac{4n^2+4n+2}{2(2n+1)}= \frac{1}{\pi} \ \frac{2n^2+2n+1}{2n+1}.
\end{eqnarray*}
Thus 
\begin{eqnarray*}
g_n^{[2]}(r)= \frac{\rho_n^{[2]}(r)}{\rho_n^{[1]}\rho_n^{[1]}} \underset{r\rightarrow 0}{\longrightarrow} \frac{n(n+1)(n^2+n+1)}{(2n^2+2n+1)^2},
\end{eqnarray*}
which concludes the proof.

\subsection{Proof of Proposition \ref{signed_rho2}}
\label{AppendixF}
Let us first start computing the covariance matrix of $Z$. 
The kernel we consider here is 
$\tilde K^n (z,w) = L_n(|z-w|^2) e^{z\bar w}$.
Putting $L_n^{(i)} = L_n^{(i)}(r^2)$, $i \geq 0$, 
we get: 
\[
\Sigma_X = \begin{pmatrix} 
\tilde K^n(0,0) & \tilde K^n(0,r) \\ 
\tilde K^n(r,0) & \tilde K^n(r,r) 
\end{pmatrix} = \begin{pmatrix} 
1 & L_n \\ 
L_n & e^{r^2}
\end{pmatrix},
\]
and thus 
\begin{eqnarray}
\label{def:Sigma_invX}
\Sigma_X^{-1} = \frac{1}{e^{r^2} - L_n^2} 
\begin{pmatrix} 
e^{r^2}& -L_n \\ 
-L_n &  1
\end{pmatrix}.
\end{eqnarray}
We then have 
\begin{eqnarray*}
\partial_{w} \tilde K^n(z,w) &=& {\mathbb E} [f_n(z) \overline{\partial_{\bar w} f_n(w)}].
\end{eqnarray*}
Therefore
\begin{eqnarray}
\label{def:Sigma_0}
\begin{aligned}
\sum\nolimits_0 = 
\begin{pmatrix}
\partial_{\bar w} \tilde K^n(0,0)&
\partial_{w} \tilde K^n(0,0)
&\partial_{\bar w} \tilde K^n(0,r)
&\partial_{w} \tilde K^n(0,r)\\
\partial_{\bar w} \tilde K^n(r,0)&\partial_{w} \tilde K^n(r,0)
&\partial_{\bar w} \tilde K^n(r,r)&\partial_{w} \tilde K^n(r,r)
\end{pmatrix}\\
=
\begin{pmatrix} 
0 & 0 & r L_n' & r L_n' \\ 
r B^{[0]} & -r L_n' & re^{r^2} & 0
\end{pmatrix},
\end{aligned}
\end{eqnarray}
in which $B^{[0]} = L_n(r^2) - L_n'(r^2)$.
Finally for $\Sigma_Y = {\mathbb E}[Y Y^H]$, we first write that: 
\begin{eqnarray*}
\partial^2_{z,w} \tilde K^n(z,w) &=& {\mathbb E} [\partial_z f_n(z) \partial_w \overline{f_n(w)}],
\end{eqnarray*}
and then, putting $B^{[i]} = L_n^{(i)} -L_n^{(i+1)}$ and remarking that $L_n (0) = 1$, $L_n'(0) = -n$, 
we get:
\begin{eqnarray}
\label{def:Sigma_Y}
\begin{aligned}
\Sigma_Y = 
\begin{pmatrix} 
\partial_{z,\bar w}^2 \tilde K^n(0,0)&
\partial_{z,w}^2 \tilde K^n(0,0)&\partial_{z,\bar w}^2 K^n(0,r)
&\partial_{z,w}^2 \tilde K^n(0,r)\\
\partial_{\bar z,\bar w}^2 \tilde K^n(0,0)&
\partial_{\bar z,w}^2 \tilde K^n(0,0)&\partial_{\bar z,\bar w}^2 \tilde K^n(0,r)
&\partial_{\bar z,w}^2 \tilde K^n(0,r)\\
\partial_{z,\bar w}^2 \tilde K^n(r,0)&\partial_{z,w }^2 \tilde K^n(r,0)
&\partial_{z,\bar w}^2 \tilde K^n(r,r)&\partial_{z,w}^2 \tilde K^n(r,r)\\
\partial_{\bar z,\bar w}^2 
\tilde K^n(r,0)&\partial_{\bar z,w}^2 \tilde K^n(r,0)&\partial_{\bar z,\bar w}^2 \tilde K^n(r,r)&
\partial_{\bar z,w}^2 \tilde K^n(r,r)
\end{pmatrix} =\\
 \begin{pmatrix} 
n+1 & 0 & r^2 B^{[1]} +B^{[0]} & r^2 B^{[1]} \\ 
0 & n & -r^2 L_n'' & -r^2 L_n''- L_n' \\ 
 r^2 B^{[1]} +B^{[0]} & - r^2 L_n'' & (n+1+r^2) e^{r^2} & 0 \\ 
r^2 B^{[1]} & - r^2 L_n''-L_n'& 0 & n e^{r^2}
\end{pmatrix}.
\end{aligned}
\end{eqnarray}

Then the variable $Z$ is zero mean Gaussian 
vector with covariance matrix 
\begin{eqnarray*}
\Gamma &= \sum\nolimits_Y - (\sum\nolimits_0)^H (\sum\nolimits_X)^{-1} \sum\nolimits_0,\\
\end{eqnarray*}
and then
$$
J_{f_n(z)}=|\partial_z f_n(z)|^2 - 
|\partial_{\bar z} f_n(z)|^2 := \omega(Z),
$$
which we rewrite as 
$$
\omega(Z) = (|z_1|^2-|z_2|^2)(|z_3|^2-|z_4|^2),
$$
putting $Z=(z_1,\dotsc,z_4)$. Then developing 
$\varphi(Z)$, we get: 
\begin{eqnarray*}
{\mathbb E}[\omega(Z)] 
= {\mathbb E}[|z_1|^2 |z_3|^2] - 
{\mathbb E}[|z_1|^2 |z_4|^2] - {\mathbb E}[|z_2|^2 |z_3|^2] + 
{\mathbb E}[|z_2|^2 |z_4|^2].
\end{eqnarray*}
As $Z$ is a Gaussian vector we have \(z_i \sim e^{i\theta} z_i\), and thus 
\begin{eqnarray*}
{\mathbb E}[z_i^2]= {\mathbb E}[e^{2i\theta} z_i^2]= e^{2i\theta}{\mathbb E}[z_i^2] = 0, \ 
{\mathbb E}[z_i z_j] = {\mathbb E}[e^{2i\theta} z_i z_j]= e^{2i\theta}{\mathbb E}[z_i z_j] = 0.
\end{eqnarray*}
Then using the Isserlis' theorem \cite{isserlis1918formula}:
\begin{eqnarray*}
{\mathbb E}[|z_i|^2 |z_j|^2] = 
{\mathbb E}[z_i \bar{z}_i z_j \bar{z}_j] 
= {\mathbb E}[z_i \bar{z}_i]{\mathbb E}[z_j \bar{z}_j]+ 
{\mathbb E}[z_i z_j]{\mathbb E}[\bar z_i \bar{z}_j]+ 
{\mathbb E}[z_i \bar z_j]{\mathbb E}[\bar z_i z_j]\\
= {\mathbb E}[|z_i|^2] {\mathbb E}[|z_j|^2] + 
|{\mathbb E}[z_i z_j]|^2 +
|{\mathbb E}[z_i \bar{z}_j]|^2
= {\mathbb E}[|z_i|^2] {\mathbb E}[|z_j|^2] + |{\mathbb E}[z_i \bar{z}_j]|^2
\end{eqnarray*}

Finally, denoting by  $\Gamma = \mathrm{Cov}(Z)$, and 
$\Omega_{ij} = {\mathbb E} [z_i z_j]$, one obtains:
\begin{eqnarray*}
{\mathbb E}[|z_1|^2 |z_3|^2] = \Gamma_{11}\Gamma_{33} + |\Gamma_{13}|^2, &
{\mathbb E}[|z_1|^2 |z_4|^2] = \Gamma_{11}\Gamma_{44} + |\Gamma_{14}|^2,\\
{\mathbb E}[|z_2|^2 |z_3|^2] = \Gamma_{22}\Gamma_{33} + |\Gamma_{23}|^2, &
{\mathbb E}[|z_2|^2 |z_4|^2] = \Gamma_{22}\Gamma_{44} + |\Gamma_{24}|^2.
\end{eqnarray*}
Then reinjecting these expressions in the conditional expectation we obtain:
\begin{eqnarray*}
{\mathbb E}[\omega(Z)] =
(\Gamma_{11} - \Gamma_{22})(\Gamma_{33} - \Gamma_{44}) 
+|\Gamma_{13}|^2 - |\Gamma_{14}|^2-|\Gamma_{23}|^2 + |\Gamma_{24}|^2.
\end{eqnarray*}

\subsection{Proof of Proposition \ref{hybrid_rho2}}
\label{AppendixG}
Using (\ref{integral_J0}), and multiplying it by $J_r$:
\begin{eqnarray*}
\left| J_0 \right| J_r = \frac{2}{\pi}\int_0^{\infty} \frac{J_r-J_r \ \cos(s J_0)}{s^2} \ ds,
\end{eqnarray*}

then 
\begin{eqnarray*}
\mathbb{E}[\left| J_0 \right| J_r] = \frac{2}{\pi}\int_0^{\infty} \frac{\mathbb{E}[J_r]-\mathbb{E}[J_r \ \cos(s J_0)]}{s^2}  \ ds.
\end{eqnarray*}
For $\mathbb{E}[J_r]$, we have:
\begin{eqnarray*}
\mathbb{E}[J_r] = \mathbb{E}[Z^* A_2 Z] = \sum_{i,j} (A_2)_{i,j}\Gamma_{i,j} = \mathrm{Tr}(\Gamma A_2),
\end{eqnarray*}
and in order to compute $\mathbb{E}[J_r \ \cos(s J_0)]$, we use the function $(s,t) \rightarrow \varphi(s,t)$ defined in (\ref{eq:def_phi}), indeed:
\begin{eqnarray*}
\partial_t \varphi(s,0) = \mathbb{E}[iJ_re^{isJ_0}],
\end{eqnarray*}
therefore 
\begin{eqnarray*}
\mathbb{E}[J_r \cos(sJ_0)] = \Im(\partial_t \varphi (s,0)).
\end{eqnarray*}

Consider $M = I-i\ \Gamma \ (s A_1 + t A_2) $ and using Jacobi's formula:
\begin{eqnarray*}
\partial_t \varphi(s,t) &=& \frac{- \det(M) \ \mathrm{Tr}( M^{-1} \ (-i) \ \Gamma A_2 )}{\det(M)^2}\\
&=& i \ \frac{  \mathrm{Tr}( M^{-1} \ \Gamma A_2 )}{\det(M)},
\end{eqnarray*}
and then one obtains:
\begin{eqnarray*}
\Im(\partial_t \varphi (s,0)) = \Re \left( \frac{ \mathrm{Tr}( M_{t=0}^{-1} \ \Gamma A_2 )}{\det(M_{t=0})}   \right),
\end{eqnarray*}
which concludes the proof.

\section*{Acknowledgement}

This research was supported by the French ANR TIPHAINE (ANR-25-CE48-3360).

\bibliographystyle{IEEEtran}
\bibliography{references.bib}

\end{document}